\documentclass[11pt,a4paper]{article}
\usepackage[T1]{fontenc}
\usepackage[utf8]{inputenc}
\usepackage{amsmath,amsthm,mathtools,mathrsfs}
\usepackage{newtxtext,newtxmath}
\usepackage[margin=26mm,headheight=14pt]{geometry}
\usepackage{microtype}
\usepackage{booktabs,array,enumitem}
\usepackage{needspace,placeins}
\usepackage{tikz}
\usetikzlibrary{arrows.meta,positioning}
\usepackage[authoryear,round]{natbib}
\usepackage{xcolor}
\definecolor{linkblue}{RGB}{25,55,95}
\usepackage[unicode,colorlinks=true,linkcolor=linkblue,citecolor=linkblue,urlcolor=linkblue]{hyperref}
\usepackage{bookmark}
\usepackage{fancyhdr}
\newtheorem{theorem}{Theorem}[section]
\newtheorem*{mainrestatement}{Theorem~\ref{thm:main} (restated)}
\newtheorem{proposition}[theorem]{Proposition}
\newtheorem{lemma}[theorem]{Lemma}
\newtheorem{corollary}[theorem]{Corollary}
\numberwithin{equation}{section}
\setlist[itemize]{leftmargin=1.4em,itemsep=3pt,topsep=5pt}
\setlist[enumerate]{leftmargin=1.8em,itemsep=5pt,topsep=5pt}
\allowdisplaybreaks[1]
\newcommand{\proofstep}[1]{\par\medskip\noindent\textbf{#1}\enspace\ignorespaces}
\makeatletter
\renewcommand{\@maketitle}{%
  \newpage\null\vskip 2em
  \begin{center}
    {\LARGE\@title\par}%
    \ifx\@author\@empty\else\vskip 1.5em{\large\@author\par}\fi
    \ifx\@date\@empty\else\vskip 1em{\large\@date\par}\fi
  \end{center}
  \par\vskip 1.5em}
\makeatother
\title{A Lower Bound for the Heavy-Ball Method\\on Smooth Convex Functions}
\author{%
\begin{tabular}{@{}>{\centering\arraybackslash}p{0.44\textwidth}@{\hspace{0.04\textwidth}}>{\centering\arraybackslash}p{0.44\textwidth}@{}}
Jianhao Ma\thanks{Equal contribution.} & Jingzhao Zhang\footnotemark[1]\\[0.35em]
\normalsize Tsinghua University & \normalsize Tsinghua University\\[0.2em]
\small\href{mailto:jianhao@tsinghua.edu.cn}{\texttt{jianhao@tsinghua.edu.cn}} &
\small\href{mailto:jingzhaoz@mail.tsinghua.edu.cn}{\texttt{jingzhaoz@mail.tsinghua.edu.cn}}
\end{tabular}}
\date{}
\hypersetup{pdftitle={A Lower Bound for the Heavy-Ball Method on Smooth Convex Functions},
pdfauthor={Jianhao Ma and Jingzhao Zhang}}
\begin{document}
\maketitle
\begin{abstract}
Can the classical Heavy-Ball method, with arbitrary horizon-dependent parameters chosen
in advance, achieve Nesterov's $O(T^{-2})$ last-iterate rate on every smooth
convex objective? We provide a negative answer. For every horizon $T\ge2$ and every
predetermined schedule with nonnegative step sizes and momenta in $[0,1)$,
there exists a convex $1$-smooth objective, with initialization distance at
most one and zero initial velocity, for which the last iterate of the Heavy-Ball method
satisfies
\[
f(x_T)-f^\star=\Omega\!\left(\frac{1}{T^\alpha\log T}\right),
\qquad \alpha=\frac{1+\sqrt5}{2}.
\]
Thus even fully nonstationary, horizon-dependent tuning cannot give
the classical Heavy-Ball method a Nesterov-rate guarantee on the smooth convex class.
\end{abstract}

\section{Introduction}\label{sec:introduction}
Nesterov's accelerated gradient method and Polyak's Heavy-Ball method both use momentum, but their convergence guarantees on general smooth convex objectives differ substantially. Nesterov's method achieves the accelerated $O(T^{-2})$ function-value rate \citep{nesterov1983,bubeck2015}. The Heavy-Ball method evaluates the gradient at the current iterate before adding a multiple of the previous displacement \citep{polyak1964}, whereas a standard representation of Nesterov's method evaluates the gradient at an extrapolated point. Whether suitable tuning can give the Heavy-Ball method the same worst-case acceleration beyond quadratic objectives has remained unclear.

We ask whether the classical Heavy-Ball method can achieve Nesterov's $O(T^{-2})$ guarantee for its actual last iterate on every smooth convex objective when its parameters are arbitrary predetermined sequences. We allow nonnegative step sizes $\eta_t$ and momenta $\beta_t\in[0,1)$. The schedule may be redesigned for every horizon, but it must be fixed before observing the objective or the trajectory.

We provide a negative answer. For every horizon $T\ge2$ and every such schedule, we construct a convex $1$-smooth objective with $\|x_0-x^\star\|\le1$ and zero initial velocity for which the last-iterate error is at least $c/(T^\alpha\log T)$, where $c>0$ is an absolute constant and $\alpha=(1+\sqrt5)/2=1.618033\ldots$ is the golden ratio. This rules out a Nesterov-rate $O(T^{-2})$ guarantee for the Heavy-Ball method under arbitrary predetermined tuning on the smooth convex class. The hard instance is a single static one-sided Huber chain, fixed before the run in dimension at most $T+1$. It adapts the sequential checkpoint construction of \citet{jung2026} to account for the persistent motion generated by momentum.

The question is nontrivial because predetermined tuning can already accelerate gradient descent. The Silver stepsize schedule achieves $O(T^{-s})$ with $s=\log_2(1+\sqrt2)=1.271553\ldots$ \citep{silver2025}, and gradient descent is the $\beta_t=0$ special case of the Heavy-Ball method. Thus predetermined schedules already improve on the classical $O(T^{-1})$ rate; the remaining question is how far this acceleration can go and whether momentum can improve on the Silver exponent.

At a high level, our proof extends checkpoint-based lower-bound constructions to trajectories with persistent momentum, separating the local effect of the step sizes from the motion that earlier gradients generate over subsequent updates.

\Needspace{18\baselineskip}
\subsection{Related work}
\paragraph{The Heavy-Ball method and acceleration.}
Polyak's classical analysis establishes an accelerated asymptotic linear
rate on strongly convex quadratics and a local rate under twice continuous
differentiability and strong convexity \citep{polyak1964}.
Under specified parameter conditions, \citet{ghadimi2015} prove an
$O(T^{-1})$ guarantee for the Ces\`aro average on smooth convex objectives
and linear convergence under strong convexity. For smooth convex coercive
objectives, \citet{sun2019} establish a last-iterate $O(T^{-1})$ rate with
suitable constant parameters, as well as linear convergence under a
restricted strong-convexity condition. For a particular iteration-dependent
schedule, \citet[Theorem~13 and Corollary~14]{sebbouh2021} prove pointwise
asymptotic last-iterate error $o(T^{-1})$ for the deterministic Heavy-Ball
method on smooth convex objectives. This does not imply a uniform
$O(T^{-1-\varepsilon})$ guarantee for any fixed $\varepsilon>0$.

Complementary lower bounds identify limits of constant-parameter tuning.
\citet{lessard2016} use integral quadratic constraints to analyze convergence
rates and exhibit a smooth strongly convex objective on which the parameters
optimal for quadratics produce a limit cycle. For arbitrary constant
parameters, \citet{goujaud2025} establish a dichotomy between slow convergence
on quadratics and cycling on a nonquadratic smooth strongly convex objective.
This rules out an accelerated worst-case asymptotic rate for the Heavy-Ball
method with constant parameters. Neither these obstructions nor the guarantees
for specified schedules settle the possibilities of arbitrary time-varying
parameters redesigned for each horizon, which are the subject of our result.

\paragraph{Predetermined gradient descent.}
Predetermined long-step schedules can improve on the classical $O(T^{-1})$
rate without requiring descent at every update \citep{grimmer2024,grimmer2023}.
For the Silver rate stated above, \citet{grimmer2025} improve the
constant.\footnote{The Silver guarantee extends to every horizon
by using the largest $n=2^k-1\le T$ and appending zero step sizes and
zero momenta; $n\ge T/2$ preserves the rate.}
Recent lower bounds show that such tuning nevertheless cannot recover
Nesterov's rate. For horizon-dependent nonnegative schedules, \citet{ma2026}
prove a lower bound $\Omega(T^{-p})$ for each
$p\in(\sqrt{2+\sqrt3},2)$. Subsequent results establish
$\Omega(T^{-\sqrt3})$ \citep{tsai2026nonanytime},
$\Omega(T^{-1.6342})$ even with negative step sizes \citep{ye2026}, and
$\Omega(T^{-\gamma})$ for positive schedules, where
$\gamma=\log_2(1+\sqrt3)$ \citep{jung2026}.

In their September~7, 2026 blog post, \citet[Theorem~1]{yeliu2026silver}
report a worst-case last-iterate lower bound
$T^{-s-C\sqrt{\log\log T/\log T}}=T^{-s-o(1)}$
for every sufficiently large horizon and every predetermined nonnegative
GD schedule, with an absolute constant $C>0$.
Together with the Silver upper bound, this identifies the optimal
polynomial exponent as $s$, while leaving a subpolynomial gap;
it does not establish the endpoint bound $\Omega(T^{-s})$.
These GD lower bounds apply only to the zero-momentum subclass, so they do
not settle what the Heavy-Ball method can achieve with predetermined
schedules. Our exponent $\alpha$ is larger than $s$, giving a numerically
weaker lower bound; our theorem, however, applies to the full admissible
class of Heavy-Ball schedules.

\paragraph{Checkpoint chains.}
\citet{jung2026} develop a sequential one-sided Huber-chain construction
for predetermined gradient descent, with selected checkpoints determining
when successive links become active. Our hard instance uses the same broad
architecture, but momentum creates two additional obstacles: earlier
gradient inputs continue to generate displacement, and a frontier coordinate
can already carry nonzero velocity when a new link is activated. We adapt
both the static chain construction and the schedule accounting to handle
these persistent-memory effects.

The key reduction uses retention normalization to remove momentum from
the local extremal step-size problem. The accumulated effect of earlier
gradients still requires separate global control. Section~\ref{sec:overview}
outlines how these two parts of the argument fit together.

\section{Main result}\label{sec:setup}
\subsection{Setting and worst-case error}
For an integer $T\ge2$, fix finite schedules
$\eta=(\eta_0,\ldots,\eta_{T-1})\in[0,\infty)^T$ and
$\beta=(\beta_0,\ldots,\beta_{T-1})\in[0,1)^T$.

The schedule is deterministic and predetermined: it cannot depend on the iterates, gradients, objective values, or the objective selected by the adversary. Different horizons may use completely different schedules. Initialize with zero velocity and run
\begin{equation}
x_{-1}=x_0,\qquad
x_{t+1}=x_t-\eta_t\nabla f(x_t)+\beta_t(x_t-x_{t-1}),
\quad t=0,\ldots,T-1.
\label{eq:P.1}
\end{equation}
The output is \(x_T\), with no averaging or subsequent processing.
Let $\mathscr F_d$ consist of the differentiable convex functions
$f:\mathbb R^d\to\mathbb R$ with a nonempty minimizer set and
$\|\nabla f(x)-\nabla f(y)\|_2\le\|x-y\|_2$ for all $x,y\in\mathbb R^d$.

Define the worst-case last-iterate error of this fixed schedule by
\begin{equation}
R_T(\eta,\beta)=
\sup_{\substack{d\ge1,\ f\in\mathscr F_d,\ x^\star\in\arg\min f\\
                    x_0\in\mathbb R^d,\ \|x_0-x^\star\|_2\le1}}
\bigl[f(x_T)-f(x^\star)\bigr].
\label{eq:P.2}
\end{equation}
When the schedule is fixed, we abbreviate $R_T(\eta,\beta)$ as $R_T$.
We write $f^\star=f(x^\star)=\min_x f(x)$ for the optimal value.
All norms are Euclidean, logarithms are natural unless a base is
indicated, and constants are absolute unless specified otherwise.

\paragraph{Order of choices.}
The order of choices is: fix the horizon and schedule, construct an
objective for that schedule, then run Heavy Ball on that fixed objective.
The step sizes need not produce descent or a stable trajectory.

\Needspace{10\baselineskip}
\subsection{Main theorem}

Let $\alpha=(1+\sqrt5)/2$ denote the golden ratio.

\begin{theorem}[Lower bound for predetermined Heavy Ball]\label{thm:main} There is an absolute constant \(c>0\) such that every horizon and schedule in \eqref{eq:P.1} admit a differentiable convex 1-smooth Huber-chain objective, with a minimizer at initialization distance at most one and dimension at most \(T+1\), for which
\begin{equation}
f(x_T)-f^\star\ge
\frac{c}{T^\alpha\log T}.
\label{eq:P.3}
\end{equation}
\end{theorem}
The bound \eqref{eq:P.3} applies to $R_T(\eta,\beta)$ and its infimum over the stated schedule class. The objective may depend on the schedule but is fixed before the run. No monotonicity or stationarity assumptions are needed. Adaptive or randomized schedules and processed outputs lie outside the theorem. For convex $L$-smooth objectives and initialization distance $r>0$, the bound scales by $Lr^2$.

\subsection{Proof roadmap}\label{sec:overview}
The proof encodes each gradient's influence on the last iterate by a
terminal weight and measures the persistence of incoming velocity by a
retention factor. Section~\ref{sec:static} realizes any prescribed
checkpoint path by one static Huber chain. Section~\ref{sec:checkpoints}
then separates checkpoint selection from path construction: a candidate list with small
aggregate gap cost contains a subset giving a large path certificate.
The same selection argument also gives a polynomial bound on the total terminal gradient weight
whenever all certificates are small.

It remains to construct such a candidate list when all certificates are small.
Section~\ref{sec:local} shows that, on a well-retained interval,
retention normalization removes momentum from the local transfer test and
reduces the problem to a one-dimensional step-size sequence. A recurrence
for this sequence yields the golden-ratio exponent; a dyadic decomposition
then bounds the step-size mass on the runs between logarithmically many
exceptional updates.
Section~\ref{sec:global} partitions the horizon into well-retained blocks
with a global memory budget and constructs the global candidate list from
these exceptional updates, the first unit-mass crossings, and the block endpoints.
Section~\ref{sec:accounting} sums the resulting gap charges to bound the
global aggregate gap cost, then combines checkpoint selection with static
realization to prove the main theorem.

\section{Static-chain certificates for Heavy Ball}\label{sec:static}
The construction has three stages. The initial displacement creates
velocity in a fresh coordinate; successive links pass that motion along
a chain; a terminal penalty converts the final motion into objective
error. We first identify the schedule weights that measure motion,
then explain how the links work, and finally derive the resulting loss
bound. All links belong to one objective fixed before the run.

\subsection{Velocity, memory multipliers, and retention}
We use $t$ for an update time, $\tau$ for an earlier reference time,
and $t'$ for a running time ($t''$ for a nested time index).
The position $x_i(t)$ is coordinate $i$ of the query $x_t$, before
update $t$. The velocity after that update is
$v_{t+1}=x_{t+1}-x_t$, with
$v_{t+1}=\beta_tv_t-\eta_t\nabla f(x_t)$ and $v_0=0$.

\paragraph{Displacement from an existing velocity.}
A unit velocity immediately after update $t$ contributes
$1,\beta_{t+1},\beta_{t+1}\beta_{t+2},\ldots$ to successive
displacements if no later gradient acts. Their sum is the
\emph{memory multiplier} $m_t$:
\begin{equation}\label{eq:memory-def}
m_t=\sum_{t'=t}^{T-1}\prod_{t''=t+1}^{t'}\beta_{t''},
\qquad m_{T-1}=1,\qquad
m_t=1+\beta_{t+1}m_{t+1}\quad(t<T-1).
\end{equation}
Empty sums are zero and empty products are one.
For a fixed coordinate $i$ and reference update $\tau$, save its
velocity as $q=(v_{\tau+1})_i=x_i(\tau+1)-x_i(\tau)$.
Its \emph{remaining displacement} under momentum alone is $m_\tau q$.
This includes update $\tau$; strictly after $x_{\tau+1}$, the remaining
contribution is $(m_\tau-1)q$.

\paragraph{Displacement created by a gradient.}
A gradient component at update $t$ creates velocity proportional to
$\eta_t$. Multiplying by its memory gives the \emph{terminal gradient
weight} $h_t=\eta_tm_t$. Thus $h_t$ is the exact coefficient of
$-\nabla f(x_t)$ in the final iterate. With
$H=\sum_{t=0}^{T-1}h_t$ the \emph{total terminal gradient weight},
unrolling the updates gives
\begin{equation}\label{eq:identities}
x_T=x_0-\sum_{t=0}^{T-1}h_t\nabla f(x_t),\qquad 1\le m_t\le T-t.
\end{equation}
\paragraph{Retention during a wait.}
For $0\le\tau\le t<T$, momentum alone changes the saved velocity
$q$ into $q\prod_{t'=\tau+1}^t\beta_{t'}$. Comparing remaining
displacement at $t$ with $m_\tau q$ cancels $q$ and gives the
\emph{retention factor}
\begin{equation}\label{eq:survival-def}
\rho_{\tau,t}=\frac{m_t}{m_{\tau}}\prod_{t'=\tau+1}^t\beta_{t'},\qquad
\rho_{\tau,t}=\rho_{\tau,t'}\rho_{t',t}\in[0,1]\quad(\tau\le t'\le t).
\end{equation}
Multiplicativity combines consecutive waits; a zero intervening
momentum erases the retained contribution.

\paragraph{Notation and scope.}
The weights $m_t,h_t$ and retention $\rho_{\tau,t}$ depend only on the
schedule. Actual remaining displacement also depends on the saved
coordinate velocity $q$. The reference time $\tau$ need not be the
first checkpoint, and the formula for $\rho_{\tau,t}$ remains defined
when $q=0$. Several coordinates may move in one update; $q$ refers
to one fixed coordinate and reference time, not a constant velocity
along the run. The bounds above use $0\le\beta_t<1$ and do not
establish the theorem for momentum greater than one.

\subsection{The sequential Huber-chain mechanism}
We adapt the one-sided Huber chain of \citet[Section~2]{jung2026}.
Each link must do two things: preserve the trajectory constructed so
far and exert a controlled force during its transfer interval. The
one-sided Huber function provides exactly these two regimes:
\[
\phi_\delta(z)=
\begin{cases}
0,&z\le0,\\
z^2/2,&0\le z\le\delta,\\
\delta z-\delta^2/2,&z\ge\delta,
\end{cases}
\qquad
\phi'_\delta(z)=
\begin{cases}
0,&z\le0,\\
z,&0<z<\delta,\\
\delta,&z\ge\delta.
\end{cases}
\]
Starting from $x_{-1}=x_0=e_1$, we use the objective
\begin{equation}\label{eq:chain}
 f(x)=\frac14\sum_{i=1}^{d-1}\phi_{\delta_i}(x_i-x_{i+1}-c_i)
       +\frac12\phi_{\delta_d}(x_d-c_d),\qquad c_i\ge0.
\end{equation}
The link terms transfer motion; the final term measures it as loss.
The function is nonnegative and vanishes at $x^\star=0$, so
$f^\star=0$ and $\|x_0-x^\star\|=1$. It is convex, differentiable,
and $1$-smooth with the displayed coefficients; the verification is
in Appendix~\ref{app:smoothness}.

\paragraph{The offset preserves the past; the threshold sets the force.}
The \emph{link margin} is $z_i=x_i-x_{i+1}-c_i$.
When $z_i\le0$, the link is \emph{inactive} and contributes zero gradient.
When $z_i\ge\delta_i$, it is \emph{saturated} and contributes gradient
$\delta_i/4$ to coordinate $i$ and $-\delta_i/4$ to coordinate $i+1$.
At update $t'$, this brakes the incoming coordinate by
$\eta_{t'}\delta_i/4$ and pushes the next by the same amount.
A larger threshold transfers motion faster but can close the margin
too soon. We therefore choose $\delta_i$ to keep the link saturated
through its intended checkpoint.

\paragraph{Extend the chain at its frontier.}
The \emph{frontier coordinate} is the end coordinate of the chain
constructed so far. For an incoming frontier $i$, all coordinates
$j>i$ are zero and at rest. Before its outgoing link is added,
coordinate $i$ receives only nonnegative velocity increments from
its left link; its position is nondecreasing. This makes it possible
to add the next link without changing earlier queries. Earlier
coordinates may continue moving: the transfer specifies the frontier,
not the time index alone.

Initially coordinate $1$ is the frontier, with position one and zero
velocity. The first link uses this initial displacement to create
velocity in coordinate $2$. Each later \emph{warm transfer} uses the
velocity already present in its incoming frontier, as follows.

\begin{enumerate}
\item \textbf{Save motion at the previous checkpoint $t_{\mathrm{prev}}$.}
Fix the incoming coordinate $i$ and save
$q=x_i(t_{\mathrm{prev}}+1)-x_i(t_{\mathrm{prev}})$.
Its remaining displacement is $m_{t_{\mathrm{prev}}}q$.
\item \textbf{Install the next link at time $\tau\ge t_{\mathrm{prev}}$.}
Choose $c_i=x_i(\tau)$. Monotonicity of the frontier and the zero
next coordinate make the new link inactive at every query through
$\tau$, preserving those updates. Its post-update velocity $q'$ at
installation satisfies
$m_\tau q'\ge\rho_{t_{\mathrm{prev}},\tau}m_{t_{\mathrm{prev}}}q$.
\item \textbf{Transfer through checkpoint $t>\tau$.}
Choose the threshold to keep the link saturated at queries
$\tau+1,\ldots,t$. The outgoing coordinate receives a positive velocity
increment at each such update. At checkpoint $t$, coordinate $i+1$
becomes the next frontier, and its remaining displacement supplies
the input to the next transfer.
\end{enumerate}
The three times satisfy $t_{\mathrm{prev}}\le\tau<t$.
Along a path they are $t_{j-1}\le\tau_j<t_j$.
Figure~\ref{fig:warm-transition} separates this chronology from the
motion-to-loss argument used next.

This installation procedure is an offline construction. We fix the
schedule, simulate successive partial chains to choose their offsets
and thresholds, then run Heavy Ball on the completed, fixed objective.
Each added link preserves the already constructed prefix.
Appendix~\ref{app:realization} supplies the saturation and response bounds.

The forces described above are exact gradients of the completed objective;
no perturbations are inserted into the updates.

\begin{figure}[!ht]
\centering
\begin{tikzpicture}[
 box/.style={draw,rounded corners=2pt,align=center,text width=3.65cm,minimum width=4cm,minimum height=1.8cm,font=\small,inner sep=5pt},
 >={Stealth}]
\node[font=\small\bfseries,anchor=west] at (-1.95,1.4)
 {(a) One warm transfer: $t_{\mathrm{prev}}\le\tau<t$};
\node[box] (previous) at (0,0)
 {\textbf{Save at $t_{\mathrm{prev}}$}\\Frontier $i$: velocity $q$\\Motion $m_{t_{\mathrm{prev}}}q$\\Coordinate $i+1$: at rest};
\node[box] (install) at (5.25,0)
 {\textbf{Install at $\tau$}\\Set $c_i=x_i(\tau)$\\New gradient is zero\\through query $\tau$};
\node[box] (transfer) at (10.5,0)
 {\textbf{Measure at $t$}\\Keep the link saturated\\at $\tau+1,\ldots,t$\\New frontier: $i+1$};
\draw[->] (previous) -- node[above,font=\small] {wait} (install);
\draw[->] (install) -- node[above,font=\small] {push} (transfer);
\node[align=center,font=\small,text width=14cm] at (5.25,-1.55)
 {Each saturated update $t'$ creates outgoing velocity $\eta_{t'}\delta_i/4$.\\
 Choose the threshold to keep the margin open through $t$.};
\node[font=\small\bfseries,anchor=west] at (-1.95,-2.55)
 {(b) From the initial displacement to terminal loss};
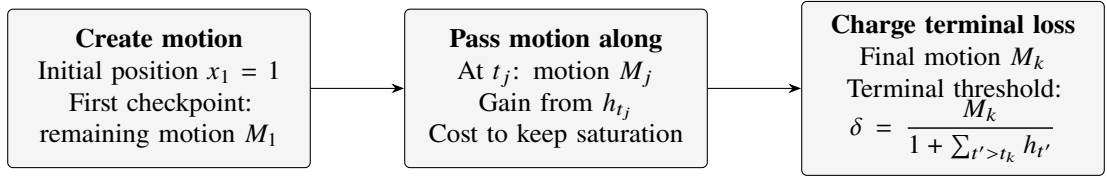
\node[box,fill=black!4,minimum height=2.1cm] (start) at (0,-4)
 {\textbf{Create motion}\\Initial position $x_1=1$\\First checkpoint:\\remaining motion $M_1$};
\node[box,fill=black!4,minimum height=2.1cm] (motion) at (5.25,-4)
 {\textbf{Pass motion along}\\At $t_j$: motion $M_j$\\Gain from $h_{t_j}$\\Cost to keep saturation};
\node[box,fill=black!4,minimum height=2.1cm] (loss) at (10.5,-4)
 {\textbf{Charge terminal loss}\\Final motion $M_k$\\Terminal threshold:\\$\displaystyle\delta=\frac{M_k}{1+\sum_{t'>t_k}h_{t'}}$};
\draw[->] (start) -- (motion);
\draw[->] (motion) -- (loss);
\node[align=center,font=\small,text width=14.5cm] at (5.25,-5.9)
 {$\displaystyle f(x_T)-f^\star\ge\frac{M_k^2}{4(1+\sum_{t'>t_k}h_{t'})}$:
 preserved motion becomes terminal loss.};
\end{tikzpicture}
\caption{The construction preserves earlier updates while creating motion in
successive coordinates. Figure~\ref{fig:warm-transition}.a shows the three times of a warm transfer;
earlier links may also contribute incoming motion. Figure~\ref{fig:warm-transition}.b shows why this
motion yields a lower bound: transfer gains accumulate, then the terminal
Huber penalty balances its strength against its braking effect.
Here $M_j$ is the remaining frontier displacement at checkpoint $t_j$,
as defined in Section~\ref{sec:motion-loss}.}
\label{fig:warm-transition}
\end{figure}
\FloatBarrier

\subsection{From one transfer to a path certificate}\label{sec:motion-loss}
We now follow the three stages in Figure~\ref{fig:warm-transition}.b.
Here $k$ is the number of checkpoints in the chosen path. Fix their
times $0\le t_1<\cdots<t_k<T$ and install every link after the first
at the preceding checkpoint, so $\tau_j=t_{j-1}$ for $j\ge2$.
The first checkpoint creates motion in coordinate $2$; checkpoint $j$
has frontier coordinate $j+1$. Thus $k$ checkpoints use $k$ links,
and the final frontier is coordinate $k+1$, not coordinate $k$.
Write $M_j$ for the \emph{remaining frontier displacement} at $t_j$:
it is $m_{t_j}$ times the velocity of coordinate $j+1$ immediately
after update $t_j$. We first consider paths with positive remaining motion.

\paragraph{Create and transfer motion.}
The first link turns the unit initial displacement into a remaining
motion satisfying
\[
M_1\ge\frac{h_{t_1}}{2(2+\sum_{t'<t_1}h_{t'})}.
\]
Each subsequent link passes a guaranteed fraction of its incoming
remaining displacement to the next coordinate:
\[
M_j\ge M_{j-1}
\frac{h_{t_j}}{2(2m_{t_{j-1}}+
                  \sum_{t_{j-1}<t'<t_j}h_{t'})}
\qquad(2\le j\le k).
\]
The numerator $h_{t_j}$ measures how effectively a gradient creates
remaining displacement. The denominator bounds the cost of choosing
a threshold that keeps the link saturated: it accounts for incoming
memory and the intervening gradient weights. Multiplying these
inequalities lower-bounds $M_k$ by the product in
\eqref{eq:immediate-path}. Appendix~\ref{app:immediate} derives both
inequalities from the exact link responses.

\paragraph{Convert the final motion into loss.}
For the last coordinate $d=k+1$, place the terminal offset at
$c_d=x_d(t_k)$. Before this penalty is added, the last coordinate is
nondecreasing, so the new penalty has zero gradient through query
$t_k$ and preserves the motion already constructed.
Its margin immediately after update $t_k$ is its
velocity, so momentum alone would give final margin $M_k$.
The terminal penalty brakes this motion with gradient at most
$\delta/2$; input from the left can only help. Thus the actual final
margin satisfies
\[
x_d(T)-c_d\ge M_k-\frac\delta2\sum_{t'>t_k}h_{t'}.
\]
Since all other objective terms are nonnegative and $f^\star=0$,
the global supporting line
$\phi_\delta(y)\ge\delta y-\delta^2/2$ gives
\[
f(x_T)-f^\star\ge\frac{\delta M_k}{2}
 -\frac{\delta^2}{4}\left(1+\sum_{t'>t_k}h_{t'}\right).
\]
A stronger terminal penalty charges more for the motion but also
brakes it more. Choosing
$\delta=M_k/(1+\sum_{t'>t_k}h_{t'})$ balances the two effects and yields
\[
f(x_T)-f^\star\ge
\frac{M_k^2}{4(1+\sum_{t'>t_k}h_{t'})}.
\]
This explains the square and the final denominator in the path bound.
\paragraph{State the resulting path bound.}
A \emph{path certificate} is the explicit error lower bound obtained
by these three stages. Write $\underline R_T$ for the \emph{certified
error lower bound}, the maximum over the finitely many path bounds
in Appendix~\ref{app:realization}, including the scalar empty path.
It satisfies $\underline R_T\le R_T$. Each maximizing path has a Huber-chain witness of the claimed dimension. The following lemma summarizes the
immediate-installation case just explained.

\begin{lemma}[Static-chain path bound]\label{lem:static}
For each schedule, the certified lower bound $\underline R_T$ is realized by
a single static Huber chain of dimension at most $T+1$, initialized at
unit distance from a minimizer. Thus $R_T\ge\underline R_T$.
For every nonempty list $0\le t_1<\cdots<t_k<T$,
\begin{equation}\label{eq:immediate-path}
\underline R_T\ge
\frac{1}{4\left(1+\sum_{t'>t_k}h_{t'}\right)}
\left[
\frac{h_{t_1}}{2\left(2+\sum_{t'<t_1}h_{t'}\right)}
\prod_{j=2}^k
\frac{h_{t_j}}
{2\left(2m_{t_{j-1}}+\sum_{t_{j-1}<t'<t_j}h_{t'}\right)}
\right]^2.
\end{equation}
The empty-path bound is $\underline R_T\ge1/[4(1+H)]$.
\end{lemma}

\paragraph{Interpretation and boundary cases.}
For a zero-motion path the bound is zero; a scalar Huber instance
covers that case without requiring a positive terminal threshold.

The quantity $M_j$ measures remaining displacement; the terminal penalty
converts it into objective error. The finite certified maximum
$\underline R_T$ retains a constructive witness, while $R_T$ is a supremum
over all admissible objectives.

Each denominator in \eqref{eq:immediate-path} measures a cost of
creating, transferring, or retaining motion. Section~\ref{sec:checkpoints}
groups these costs so that we can choose a path with a large error
lower bound.

Setting $\beta=0$ gives $m_t=1$ and $h_t=\eta_t$, so
\eqref{eq:immediate-path} recovers the GD checkpoint certificate of
\citet[Lemma~2.3]{jung2026}.

\section{Weighted checkpoint selection}\label{sec:checkpoints}
Equation~\eqref{eq:immediate-path} already gives a loss lower bound
for any chosen checkpoint list. The question here is which list to
choose. We first explain how a selection changes the constructed chain,
then group its denominator costs in a way that supports selection.

\paragraph{Choose checkpoints, then construct their chain.}
A \emph{candidate} is an available update time; a \emph{checkpoint}
is a candidate actually selected. Let $L$ be the number of candidates,
with times $0\le t_1<\cdots<t_L<T$. Selecting $k$ of them gives
$t_{i_1}<\cdots<t_{i_k}$, where $0\le k\le L$.
For a nonempty selection, Lemma~\ref{lem:static} constructs $k$ links
and a terminal penalty in dimension $k+1$. Selection position $j$
determines frontier coordinate $j+1$; candidate index $i_j$ identifies
its checkpoint time.

Consider three candidate times $t_1<t_2<t_3$.
If all three are selected, the first link creates motion in coordinate
$2$ by $t_1$. A link installed at $t_1$ transfers motion to coordinate
$3$ by $t_2$; another installed at $t_2$ transfers it to coordinate
$4$ by $t_3$. The terminal penalty acts on coordinate $4$.
If only $t_1,t_3$ are selected, the link installed at $t_1$ transfers
motion directly to coordinate $3$ by $t_3$, and the terminal penalty
acts there. At $t_2$ there is no checkpoint transition or new link:
the existing transfer continues.

Skipping a candidate therefore skips a link transition, not an
algorithm update. All $T$ updates still occur. The shorter path has
its own offsets and thresholds, chosen by the static construction to
keep its transfer saturated through $t_3$; it is not obtained by simply
deleting a link from the previously constructed objective.

\paragraph{Apply the path bound once to the selected list.}
For the selection $(t_1,t_3)$, substitute this two-checkpoint list
directly into \eqref{eq:immediate-path}. The constructed objective has
loss at least
\[
\frac{1}{4(1+\sum_{t'>t_3}h_{t'})}
\left[
\frac{h_{t_1}}{2(2+\sum_{t'<t_1}h_{t'})}
\frac{h_{t_3}}{2(2m_{t_1}+\sum_{t_1<t'<t_3}h_{t'})}
\right]^2.
\]
The skipped weight $h_{t_2}$ now appears in the intervening sum in
the transfer denominator. There is no separate numerator factor at
$t_2$. Equation~\eqref{eq:immediate-path} is applied once to this
entire path; its proof has already multiplied the individual transfer
guarantees.

\paragraph{Assign costs before choosing the subset.}
We want one list of costs that covers the denominators for every
selection. For each region before, between, or after the candidates,
define the \emph{primitive gap cost weight} $g_j$ by
\begin{equation}\label{eq:embed}
\begin{aligned}
g_0&=2+\sum_{t<t_1}h_t,\\
g_j&=1+2m_{t_j}+\sum_{t_j<t<t_{j+1}}h_t\quad(1\le j<L),\\
g_L&=1+2m_{t_L}+\sum_{t>t_L}h_t.
\end{aligned}
\end{equation}
Each cost combines intervening terminal gradient weights with memory
and constant allowances. It bounds a denominator cost, rather than
specifying an offset or threshold. The thresholds in the construction
ensure the transfer guarantees; the gap weights summarize their costs
using schedule quantities alone.

\paragraph{Merge the costs around a skipped candidate.}
Return to the selection $(t_1,t_3)$. Combine the two gaps adjacent to
$t_2$ and its skipped weight into $g_1+h_{t_2}+g_2$. Expanding the
definitions gives
\[
\begin{aligned}
g_1+h_{t_2}+g_2
&=\left(2m_{t_1}+\sum_{t_1<t'<t_3}h_{t'}\right)
  +(2+2m_{t_2})\\
&\ge 2m_{t_1}+\sum_{t_1<t'<t_3}h_{t'}.
\end{aligned}
\]
The first parentheses are exactly the transfer cost needed by the
shorter path; the remaining terms are harmless extra allowances.
Also $g_0=2+\sum_{t'<t_1}h_{t'}$ and
$g_3\ge1+\sum_{t'>t_3}h_{t'}$. Replacing the actual denominator costs
by these larger quantities weakens the bound but preserves its validity:
\[
f(x_T)-f^\star\ge
\frac1{4g_3}
\left[
\frac{h_{t_1}}{2g_0}
\frac{h_{t_3}}{2(g_1+h_{t_2}+g_2)}
\right]^2.
\]
This is what \emph{merging} means: adding denominator allowances,
not merging coordinates or changing the update rule.

Selecting all three candidates instead yields the product
$[h_{t_1}/(2g_0)][h_{t_2}/(2g_1)][h_{t_3}/(2g_2)]$,
squared and divided by $4g_3$. Thus selection trades a numerator
factor against the cost of absorbing that candidate into a gap.
Neither choice is uniformly better. For a general subset, absorb each
omitted candidate weight and its neighboring gaps into the region
between retained checkpoints, or into the initial or final region.
Appendix~\ref{app:capacity} denotes these merged costs by $s_j$ and
checks that they cover every denominator in
\eqref{eq:immediate-path}, including the empty selection.

\paragraph{Guarantee that some selection has a large bound.}
The selection theorem depends on the primitive gaps through the
\emph{aggregate gap cost} $\mathcal D$:
\begin{equation}\label{eq:gap-cost}
\mathcal D(t_1,\ldots,t_L)
=\left(\sum_{j=0}^L g_j^{2/3}\right)^{3/2}.
\end{equation}
For no candidates, set $g_0=2+H$ and $\mathcal D(\varnothing)=2+H$.
The following proposition says that some selection, realized by its
own static chain, has a certificate at least the reciprocal of this
aggregate cost, up to an absolute constant.

\begin{proposition}[Checkpoint selection]\label{prop:embed}
Every candidate list satisfies
\begin{equation}\label{eq:embedbound}
\underline R_T\ge\frac{1}{8\mathcal D(t_1,\ldots,t_L)}.
\end{equation}
\end{proposition}
Appendix~\ref{app:capacity} proves this guarantee by a recursion over
candidate lists. The $2/3$ exponent is the one for which its two-child
cost inequality \eqref{eq:power} closes. This recursion finds a good
selection; it does not repeatedly apply the path bound to one trajectory.
After selection, \eqref{eq:immediate-path} and the merged-cost comparison
supply the objective and loss bound.

\paragraph{A large gradient weight would give a large certificate.}
Before constructing candidates, we extract a coarse consequence of the
same selection argument. Make every update a candidate. A very large
terminal gradient weight $h_t$ can then be selected as a checkpoint,
with suitable selections on either side to control its denominator
costs. Equation~\eqref{eq:immediate-path} turns that full selected path
into a loss certificate. If all certificates are at most one, no
individual weight can be too large.

\begin{corollary}[Polynomial mass]\label{cor:mass}
If $\underline R_T\le1$, then $H\le144T^6$.
\end{corollary}
\begin{proof}[Proof sketch]
With every update a candidate, the candidate weights are $h_t$ and
there are no intervening updates in the primitive gaps. Thus $g_0=2$
and the other gap costs are $1+2m_t=O(T)$. Each left or right sublist
has at most $T+1$ gaps and aggregate cost $O(T^{5/2})$.

Fix a positive weight $h_t$ and require its checkpoint to be selected.
The selection proof chooses checkpoints on both sides: its one-power
cost factors as $2/h_t$ times the optimized left and right costs.
Each side is bounded by its aggregate gap cost, independently of how
large its candidate weights are. The condition $\underline R_T\le1$
forces the normalized cost to be at least one; the full proof adds a
unit to the final gap to pass from the squared certificate to this
one-power cost. Consequently,
\[
h_t\le2\bigl[3T(T+1)^{3/2}\bigr]^2\le144T^5.
\]
Summing over the $T$ updates proves the claim.
Appendix~\ref{app:mass} gives the cost factorization and normalization
in detail. Selecting the large update alone need not suffice: the
surrounding selections keep other large weights from inflating its gaps.
\end{proof}

\paragraph{Interpretation of the memory cost.}
The $2m_{t_j}$ term accounts for the incoming remaining displacement.
For an adjacent transfer from $\tau$ to $\tau+1$ with incoming velocity
$q>0$, the choice $\delta=q$ creates outgoing velocity
$\eta_{\tau+1}q/4$. Its remaining-displacement ratio is
\[
\frac{h_{\tau+1}q/4}{m_\tau q}
=\frac{h_{\tau+1}}{2(2m_\tau)}.
\]
An extra step-size factor is unnecessary: earlier updates have already
created the saved velocity $q$. For longer transfers, the intervening
weight sum enters as in \eqref{eq:immediate-path}. The extra $1$ in
interior gap costs and the extra $2m_{t_L}$ in the final gap cost are
allowances that permit gaps to be combined by addition.

The corollary controls $h_t=\eta_tm_t$, not just the step size $\eta_t$.
It imposes no stability assumption: when $\underline R_T>1$, a hard
instance is already available. Its polynomial bound on $H$ will make
the local product budget polynomial and the number of dyadic levels
$O(\log T)$. The remaining task is to construct candidates with small
$\mathcal D$, using local step-size control and the global memory budget.

\section{Small certificates force local step-size control}\label{sec:local}
Throughout this section assume $\underline R_T\le1$, so
Corollary~\ref{cor:mass} gives $H\le144T^6$.
Section~\ref{sec:checkpoints} bounded individual weights by selecting
a large one. We now constrain products of local transfer gains:
a large product would likewise give a large path certificate.
This stronger local control will bound step-size mass between
exceptional updates and identify the candidates used in
Section~\ref{sec:global}.

\Needspace{11\baselineskip}
\subsection{Removing momentum from one local transfer}
A warm-transfer gain depends on both step sizes and momentum.
The key reduction is to isolate the retention loss, leaving two
constraints that can be controlled by step sizes alone: generating
outgoing motion and keeping the link saturated. The following score
takes the smaller of the corresponding bounds.
For $\tau<t$, define the \emph{step-size transfer score} $\psi(\tau,t)$:
\begin{equation}\label{eq:flat}
\psi(\tau,t)=\min\left\{
\frac14\sum_{t'=\tau+1}^t\eta_{t'},\quad
\frac{(t-\tau)\sum_{t'=\tau+1}^t\eta_{t'}}
{2\left(2+\sum_{t'=\tau+1}^t(t-t')\eta_{t'}\right)}
\right\}.
\end{equation}
The first branch measures available step-size mass; the second limits
the force so that the link remains saturated through query $t$.
To see both constraints, consider an undamped comparison with incoming
velocity $q>0$. A saturated link of threshold $\delta$ changes relative
velocity by $-\eta_{t'}\delta/2$ at update $t'$. That change contributes
to $t-t'$ displacements before query $t$, so the comparison margin is
\[
(t-\tau)q-\frac{\delta}{2}\sum_{t'=\tau+1}^t(t-t')\eta_{t'}.
\]
First-query saturation requires $\delta\le q$; final-query saturation
requires this margin to be at least $\delta$. Together they give
\[
\frac{\delta}{q}\le
\min\left\{1,\frac{2(t-\tau)}{2+\sum_{t'=\tau+1}^t(t-t')\eta_{t'}}\right\}.
\]
Multiplying by $\sum_{t'=\tau+1}^t\eta_{t'}/4$, the outgoing velocity per unit
threshold, yields the two branches of \eqref{eq:flat}. This comparison
motivates the score; retention normalization below handles the actual
momentum schedule.

Write $G(\tau,t)$ for the certified warm-transfer gain defined in
Appendix~\ref{app:realization}. On a positive-retention interval,
\[
\frac{G(\tau,t)}{\rho_{\tau,t}}\ge\psi(\tau,t).
\]
To see why momentum can disappear, track the fraction of incoming
velocity retained at each update. These fractions are nonincreasing.
Normalization weights each gradient input by the reciprocal of its
retained fraction, so the outgoing response is at least the original
step-size mass. The same monotonicity bounds the saturation constraint
through ratios of earlier and later retained fractions. Both comparisons
therefore use the same original step sizes on every subinterval.
The following lemma combines these local gains along a path.
Its subscripts name the start and end of the interval and the time
$t_{\mathrm{cross}}$ when the required unit-mass prefix is available.

\Needspace{9\baselineskip}
\begin{lemma}[Momentum removal]\label{lem:flat}
Recall that $H=\sum_{t'=0}^{T-1}h_{t'}=\sum_{t'=0}^{T-1}\eta_{t'}m_{t'}$
is the total terminal gradient weight over the full horizon.
Suppose $0\le t_{\mathrm{start}}\le t_{\mathrm{cross}}\le t_{\mathrm{end}}<T$, $\rho_{t_{\mathrm{start}},t_{\mathrm{end}}}\ge1/2$, and
$\sum_{t'=t_{\mathrm{start}}}^{t_{\mathrm{cross}}} h_{t'}\ge1$. Every nonempty chronological family
$t_{\mathrm{cross}}\le \tau_1<t_1\le\cdots\le \tau_k<t_k\le t_{\mathrm{end}}$ satisfies
\begin{equation}\label{eq:delayed-path}
\underline R_T\ge
\frac{\left(\prod_{i=1}^k\psi(\tau_i,t_i)\right)^2}
{256(H+2)^2(1+H)}.
\end{equation}
\end{lemma}
\paragraph{Retain motion across the whole path.}
The selected path lies inside the assumed interval. Multiplicativity gives
\[
\rho_{t_{\mathrm{start}},t_{\mathrm{end}}}
=\rho_{t_{\mathrm{start}},\tau_1}\,
 \rho_{\tau_1,t_k}\,
 \rho_{t_k,t_{\mathrm{end}}}
\le\rho_{\tau_1,t_k},
\]
because each outer factor lies in $[0,1]$. Therefore
$\rho_{\tau_1,t_k}\ge1/2$. This conclusion comes from the interval
hypothesis, not merely from each individual transfer having positive
retention.

\paragraph{Combine the transfer gains.}
For each transfer, the local comparison gives
$G(\tau_i,t_i)\ge\rho_{\tau_i,t_i}\psi(\tau_i,t_i)$.
The path also pays for the waits between transfers. All these retention
factors multiply to exactly
\[
\rho_{\tau_1,t_1}\prod_{i=2}^k
\bigl(\rho_{t_{i-1},\tau_i}\rho_{\tau_i,t_i}\bigr)
=\rho_{\tau_1,t_k}\ge\frac12.
\]
Thus the product of step-size scores retains one overall factor at
least $1/2$ in motion, or $1/4$ after squaring for loss. Bounding each
transfer separately by $1/2$ would instead give an unnecessarily weak
factor depending on the number of transfers. Without the interval
hypothesis, the exact product could be close to zero, so large
$\psi$-scores alone would not give \eqref{eq:delayed-path}.

The unit-mass prefix creates the initial motion, and the terminal Huber
penalty converts the resulting motion into loss. Their costs are
bounded by $H$, which explains its appearance in the denominator.
Momentum has therefore not disappeared entirely: it enters through
$H=\sum\eta_{t'}m_{t'}$ and through which intervals satisfy the retention
hypothesis. Appendix~\ref{app:flat} combines these bounds to obtain
the constant in \eqref{eq:delayed-path}.

\paragraph{Choosing intervals for the lemma.}
The schedule fixes every $\rho_{\tau,t}$; the construction cannot choose
its value. The hypothesis
$\rho_{t_{\mathrm{start}},t_{\mathrm{end}}}\ge1/2$ says that at least
half the remaining displacement survives momentum propagation across
this particular interval. It need not hold across the whole horizon.
For example, an intervening zero momentum makes retention zero.

We choose intervals on which the hypothesis holds, and then use the
lemma for paths inside them. Section~\ref{sec:global} partitions the
schedule into such blocks, extending each backward only as far as the
retention condition permits. When a unit-mass crossing exists, it supplies the other hypothesis;
the local paths lie after that crossing. Blocks without such a crossing
are handled separately in the candidate construction. Thus this
is part of selecting where to construct certificates, not a restriction
imposed on the allowed momentum schedules. The choice $1/2$ provides
a fixed positive retention budget; its role here is to prevent the
motion bound from becoming arbitrarily small.

\paragraph{Interpreting the time weights.}
For a four-update interval, the denominator's weighted sum is
$3\eta_{\tau+1}+2\eta_{\tau+2}+\eta_{\tau+3}+0\eta_{\tau+4}$.
The last step creates outgoing velocity but cannot change the query
that precedes it. These weights count accumulated displacement before
the local query, not forgetting of momentum or a terminal memory
multiplier moved inside the sum.

\subsection{The resulting admissible-sequence problem}
Under Lemma~\ref{lem:flat}'s hypotheses, $\underline R_T\le1$ bounds
every chronological product after the unit-mass prefix by the
\emph{product budget} $\Lambda=16(H+2)\sqrt{1+H}$. The remaining question is
one-dimensional: how much local step-size mass can obey this restriction?

Let $N$ be the \emph{number of updates in the local sequence},
reindexed as $\eta_1,\ldots,\eta_N$. Use the same score \eqref{eq:flat}
for $0\le \tau<t\le N$. Call it \emph{$\Lambda$-admissible}, for
$\Lambda\ge1$, if every chronological disjoint interval family satisfies
\[
\prod_{i=1}^k\psi(\tau_i,t_i)\le\Lambda,
\qquad 0\le \tau_1<t_1\le\cdots\le \tau_k<t_k\le N.
\]
The empty family has product one. To measure how much step-size mass
this restriction permits at a given length, define the
\emph{maximum cumulative step size} $\mathcal S_\Lambda(N)$:
\begin{equation}\label{eq:product-budget}
\mathcal S_\Lambda(N)
=\sup\left\{\sum_{i=1}^N\eta_i:
\eta\in[0,\infty)^N\text{ is $\Lambda$-admissible}\right\}.
\end{equation}
A singleton interval has score $\eta_i/4$, so
$\mathcal S_\Lambda(N)\le4\Lambda N$. Restriction to a contiguous
subsequence preserves admissibility.

\paragraph{Meaning of the sequence quantities.}
When applied to a run of the original schedule, $N$ is that run's
length, with $N\le T$; it is separate from the error $R_T$.
For one sequence, $\sum_{i=1}^N\eta_i$ is its cumulative step size
(the discrete analogue of a step-size integral). The supremum
$\mathcal S_\Lambda(N)$ bounds this sum over all sequences obeying
the product budget; it is not an objective error.

\Needspace{11\baselineskip}
\subsection{The local recurrence and the golden ratio}
We now work with an abstract nonnegative sequence and the specific
score $\psi$ in \eqref{eq:flat}. The recurrence below uses only that
formula and 2-admissibility; it does not use an objective, a Huber chain,
or a Heavy-Ball trajectory.
We first solve the sequence problem at product budget two; a dyadic
decomposition will later reduce the larger budget to this case.
To bound the mass of a long 2-admissible run, compare it with shorter
windows. Here $\ell$ is an \emph{integer window length}, measured in
updates: a window starting after time $\tau$ contains entries
$\tau+1,\ldots,\tau+\ell$. The following recurrence
holds for every allowed $\ell$ and separates an explicit $N/\ell$ cost
from a term controlled by the same extremal problem at length $\ell$.
\begin{lemma}[Local recurrence]\label{lem:recurrence}
For integers $2\le\ell\le N/2$,
\begin{equation}\label{eq:energyrec}
\mathcal S_2(N)\le64N/\ell+10\ell\mathcal S_2(\ell).
\end{equation}
\end{lemma}
The score $\psi$ contains ordinary step-size mass and age-weighted mass.
Weight a sliding window of length $\ell$ by $i(\ell-i)$: the discrete
derivative of these parabolic weights exposes precisely those two
masses. Above a fixed baseline, excess weighted mass contracts unless
the best chronological product in a prefix increases. For a
2-admissible sequence, this prefix product lies between one and two,
so its total increase is at most one. Summing the drift spends this
bounded increase across overlapping windows; converting weighted mass
back to ordinary mass and treating the boundary windows gives
\eqref{eq:energyrec}. The baseline contributes the $N/\ell$ term,
while the drift and boundary estimates contribute the term involving
$\ell\mathcal S_2(\ell)$. Appendix~\ref{app:drift} gives the exact
identities and constants.

The recurrence leaves a choice of window length: larger windows reduce
the explicit cost but increase the recursive term. To find a
self-consistent power law, suppose
$\mathcal S_2(\ell)\lesssim\ell^{\alpha-1}$. The two terms then scale
as $N/\ell$ and $\ell^\alpha$. Balancing these competing costs gives
\[
\mathcal S_2(N)\lesssim N/\ell+\ell^\alpha,
\qquad N/\ell\asymp\ell^\alpha
\quad\Longrightarrow\quad
\ell\asymp N^{1/(1+\alpha)}.
\]
Both terms now have exponent $\alpha/(1+\alpha)$ in $N$.
For the recurrence to reproduce the assumed power law at length $N$,
this output exponent must match the input exponent $\alpha-1$.
Self-consistency therefore requires
\begin{equation}\label{eq:golden-balance}
\alpha-1=\frac{\alpha}{1+\alpha}
\quad\Longleftrightarrow\quad
\alpha^2-\alpha=1
\quad\Longrightarrow\quad
\alpha=\frac{1+\sqrt5}{2}.
\end{equation}
\begin{lemma}[Local step-size mass]\label{lem:local}
With $C_0=2^{15}$, every 2-admissible sequence of length $N\ge1$ has
total step size at most $C_0N^{\alpha-1}$; equivalently,
$\mathcal S_2(N)\le C_0N^{\alpha-1}$.
\end{lemma}
Strong induction with window length of order $N^{2-\alpha}$ makes the
balance rigorous. Appendix~\ref{app:local-induction} checks integer
rounding and the constant $C_0$.

\paragraph{Role of the window length and exponent.}
The run length $N$ is given, while the integer window length $\ell$
is chosen in the proof to balance the recurrence. Neither changes
the algorithm's step sizes or momentum.
The golden ratio is thus the fixed point of this proof recurrence,
not a claim about the optimal exponent for Heavy Ball.
The two branches of $\psi$ are essential: they expose the ordinary
and displacement-weighted sums used in the sliding-window identity.
Equation~\eqref{eq:energyrec} is therefore an algebraic consequence of
this particular score and its product constraint, not a recurrence for
an arbitrary score. The Huber-chain argument is used earlier to show
that the schedule yields admissible sequences, and later to convert
the resulting schedule bounds into a loss lower bound.

\subsection{From a large product budget to 2-admissible runs}
The local mass bound applies to product budget two. To use it under a
larger budget $\Lambda$, we isolate the updates where the best prefix
product crosses a power of two. Between crossings, every chronological
product has budget less than two. The number of crossings, rather than
the size of $\Lambda$ itself, is what enters the decomposition.

\paragraph{Control products by a prefix potential.}
For a local sequence of length $N$, let $\Pi_t$ be the largest product
of $\psi$-scores over chronological disjoint intervals contained in
its first $t$ entries. The empty product is one, and admissibility gives
\[
\Pi_0=1,\qquad 1\le\Pi_\tau\le\Pi_t\le\Lambda
\quad(0\le\tau\le t\le N).
\]
Appending interval $(\tau,t]$ to a best prefix family gives
$\Pi_t\ge\Pi_\tau\psi(\tau,t)$. More generally, appending any
chronological interval family inside $(\tau,t]$ gives
\[
\prod_i\psi(\tau_i,t_i)\le\frac{\Pi_t}{\Pi_\tau}.
\]
This endpoint ratio is why controlling the growth of $\Pi$ controls
all products within a run, not just scores of individual intervals.
The exact prefix recursion is recorded in \eqref{eq:potential}.

\paragraph{Cut whenever the dyadic level increases.}
Assign prefix $t$ the integer level $\lfloor\log_2\Pi_t\rfloor$.
Remove entry $t$ from the local sequence whenever
\[
\lfloor\log_2\Pi_t\rfloor>
\lfloor\log_2\Pi_{t-1}\rfloor.
\]
The levels never decrease. They start at zero and cannot exceed
$\lfloor\log_2\Lambda\rfloor$, so there are at most that many
removed entries. A single entry can cross several levels at once;
it is still removed only once. These removals separate contiguous
runs. They do not delete updates from the Heavy-Ball trajectory;
Section~\ref{sec:global} uses the removed entries as candidates.

For a remaining run with entries $\tau+1,\ldots,t$, no level changes
between prefixes $\tau$ and $t$. Both potentials lie in the same band
$[2^j,2^{j+1})$, so
\[
\frac{\Pi_t}{\Pi_\tau}<2.
\]
The endpoint-ratio bound makes the entire run 2-admissible.
Crucially, after a removed entry $\tau$, the next run uses
$\Pi_\tau$ as its baseline: it does not pay for the jump that was
just removed. Figure~\ref{fig:dyadic-cuts} illustrates this construction.

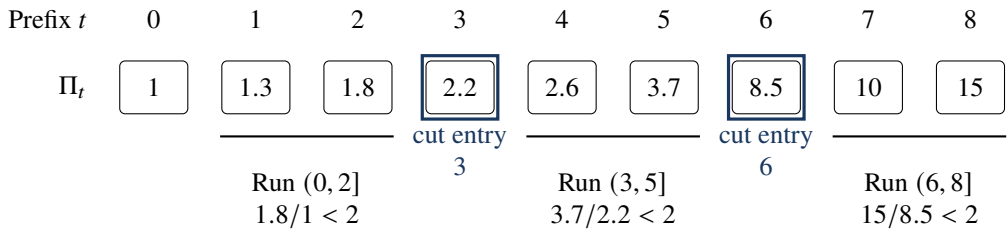
\begin{figure}[!ht]
\centering
\begin{tikzpicture}[x=1.35cm,y=1cm,font=\small]
\node[anchor=east] at (-0.55,0.9) {Prefix $t$};
\node[anchor=east] at (-0.55,0) {$\Pi_t$};
\foreach \t/\v in {0/1,1/1.3,2/1.8,3/2.2,4/2.6,5/3.7,6/8.5,7/10,8/15} {
 \node at (\t,0.9) {$\t$};
 \node[draw,rounded corners=2pt,minimum width=9mm,minimum height=7mm] at (\t,0) {$\v$};
}
\foreach \t in {3,6} {
 \draw[linkblue,very thick] (\t-0.37,-0.42) rectangle (\t+0.37,0.42);
 \node[linkblue,align=center] at (\t,-0.85) {cut entry\\$\t$};
}
\draw[thick] (0.65,-0.65) -- (2.35,-0.65);
\draw[thick] (3.65,-0.65) -- (5.35,-0.65);
\draw[thick] (6.65,-0.65) -- (8.35,-0.65);
\node[align=center] at (1.5,-1.5) {Run $(0,2]$\\$1.8/1<2$};
\node[align=center] at (4.5,-1.5) {Run $(3,5]$\\$3.7/2.2<2$};
\node[align=center] at (7.5,-1.5) {Run $(6,8]$\\$15/8.5<2$};
\end{tikzpicture}
\caption{Schematic dyadic decomposition of a prefix potential bounded
by $\Lambda=16$. Entry $3$ crosses threshold $2$, while entry $6$ crosses
both thresholds $4$ and $8$ in one jump. Removing these two entries leaves
three runs whose endpoint ratios are below two. Each new run is
measured from the potential immediately after its preceding cut.}
\label{fig:dyadic-cuts}
\end{figure}

\begin{lemma}[Dyadic decomposition]\label{lem:bands}
Removing at most $\lfloor\log_2\Lambda\rfloor$ entries from a
$\Lambda$-admissible sequence leaves at most
$1+\lfloor\log_2\Lambda\rfloor$ contiguous 2-admissible runs.
\end{lemma}
The logarithmic dependence is a count of multiplicative levels:
$1,2,4,8,\ldots,\Lambda$. It bounds the number of exceptional entries
and controlled runs, not the total mass of an arbitrary
$\Lambda$-admissible sequence. The exceptional entries may carry
large step sizes; the later checkpoint construction accounts for them
separately. Appendix~\ref{app:bands} gives the formal proof.

\paragraph{Applying the decomposition to the actual schedule.}
The momentum-removal bound \eqref{eq:delayed-path}, together with
$\underline R_T\le1$, gives
\[
\prod_i\psi(\tau_i,t_i)
\le16(H+2)\sqrt{1+H}=\Lambda.
\]
This defines the budget for the local sequence. Corollary~\ref{cor:mass}
does not choose $\Lambda$; it bounds $H\le144T^6$, and hence
\[
\Lambda\le32768T^9,
\qquad \log_2\Lambda\le15+9\log_2T.
\]
Appendix~\ref{app:accounting-constants} verifies these constants.
Thus a decomposition using $O(\log\Lambda)$ cuts will use only
$O(\log T)$ cuts for the actual schedule.

\section{Constructing the global candidate list}\label{sec:global}
To apply Section~\ref{sec:local} across the whole horizon, we first
partition the update times into blocks that retain enough motion.
Within each block, a prefix creates initial motion; the remaining
suffix then satisfies the product bound needed for admissibility.
Dyadic decomposition gives runs with controlled step-size mass, which
we convert into terminal-weight bounds. Throughout this section assume
$\underline R_T\le1$.

\subsection{Retention blocks and their memory budget}
Call a block $I=[t_{\mathrm{start}},t_{\mathrm{end}}]$
\emph{well-retained} if
\[
 \rho_{t_{\mathrm{start}},t_{\mathrm{end}}}
 =\frac{m_{t_{\mathrm{end}}}}{m_{t_{\mathrm{start}}}}
   \prod_{t'=t_{\mathrm{start}}+1}^{t_{\mathrm{end}}}\beta_{t'}
 \ge\frac12.
\]
Thus at least half the remaining displacement attributable to a saved
velocity survives momentum propagation across the block. The memory
ratio matters: $\rho$ measures remaining displacement, rather than
velocity alone. This endpoint condition guarantees the same bound on
every subinterval, since
\[
 \frac12\le\rho_{t_{\mathrm{start}},t_{\mathrm{end}}}
 =\rho_{t_{\mathrm{start}},\tau}\rho_{\tau,t}
   \rho_{t,t_{\mathrm{end}}}\le\rho_{\tau,t}
 \qquad(t_{\mathrm{start}}\le\tau\le t\le t_{\mathrm{end}}).
\]
Here the outer factors lie in $[0,1]$. Consequently, retention factors
along any chronological path inside the block telescope to a single
factor at least $1/2$, as required by Lemma~\ref{lem:flat}.

The following partition supplies well-retained blocks and bounds
their largest memory multipliers $d_I$, which will convert step-size
mass into terminal-weight mass.
\begin{lemma}[Backward-greedy blocks]\label{lem:blocks}
The update indices have a partition into consecutive blocks $I$ such that $\rho_{\tau,t}\ge1/2$ whenever $\tau\le t$ lie in one block, and
\begin{equation}\label{eq:memory}
 \sum_I d_I\le3T,\qquad d_I:=\max_{t\in I}m_t.
\end{equation}
\end{lemma}
\begin{proof}[Proof sketch]
Figure~\ref{fig:candidate-block}.a illustrates the block construction.
We construct the blocks backward from update $T-1$. For each right
endpoint, we extend the block as far left as possible while keeping
its endpoint retention at least $1/2$, and then repeat on the remaining
prefix. The subinterval argument above guarantees the required
retention throughout each block.

To bound the total memory, consider two consecutive block endpoints
$t_{\mathrm{end}}<t_{\mathrm{next}}$. Maximality of the later block
implies $\rho_{t_{\mathrm{end}},t_{\mathrm{next}}}<1/2$: otherwise it
could have included the preceding endpoint. At least half of the
remaining displacement at $t_{\mathrm{end}}$ is therefore expended
before $t_{\mathrm{next}}$. Since each intervening update contributes
at most one unit to the memory multiplier, we obtain
\[
 \frac{m_{t_{\mathrm{end}}}}2
 \le m_{t_{\mathrm{end}}}
       (1-\rho_{t_{\mathrm{end}},t_{\mathrm{next}}})
 \le t_{\mathrm{next}}-t_{\mathrm{end}}.
\]
The lower bracket in Figure~\ref{fig:candidate-block}.a marks the
interval between consecutive endpoints used in this bound.
These time intervals are disjoint. Summing the bound and including the
last endpoint, where $m_{T-1}=1$, bounds the sum of endpoint memories
by $2T$.

Finally, the recurrence $m_t\le1+m_{t+1}$ implies that the maximum
memory in a block exceeds its endpoint memory by at most $|I|-1$.
The block lengths sum to $T$, so these additional contributions total
at most $T$. Together with the endpoint bound, this gives
$\sum_I d_I\le3T$. Appendix~\ref{app:blocks} gives the full proof.
\end{proof}

\paragraph{Role in the candidate construction.}
Figure~\ref{fig:candidate-block}.b zooms into one of these blocks.
Its top bracket records the two guarantees supplied by the lemma:
retention is at least $1/2$, and every memory multiplier is at most
$d_I$. Retention allows us to apply Lemma~\ref{lem:flat} after the first
unit-mass candidate. The memory bound then converts each blue run's
step-size bound into the terminal-weight bound displayed inside it,
using $h_t=\eta_tm_t\le d_I\eta_t$. The next subsection constructs
the candidate dots and the runs between them. When we later sum their
costs over all blocks, the budget $\sum_I d_I\le3T$ controls the total
of these conversion factors. The schedule fixes the retention factors;
the construction chooses the block boundaries and candidate times.

\subsection{Building candidates inside each block}\label{sec:block-candidates}
We now select candidate updates within each retention block so that
the total terminal weight between candidates can be bounded. We first
mark the earliest update at which the cumulative terminal weight in
the block reaches one. When this crossing exists, the prefix supplies
the initial motion required by Lemma~\ref{lem:flat}, which makes the
remaining suffix $\Lambda$-admissible. We then mark the suffix's dyadic
crossings as additional candidates; the runs between these marked
updates are 2-admissible and therefore have controlled step-size mass.
Finally, we mark the block's right endpoint. If the block never reaches
unit terminal weight, we mark only that endpoint, since the whole
block already has terminal weight less than one. This construction
partitions the updates for the analysis; it does not remove updates
from the algorithm's schedule.

For the precise construction, consider a block
$I=[t_{\mathrm{start}},t_{\mathrm{end}}]$ from Lemma~\ref{lem:blocks},
where $t_{\mathrm{start}}$ and $t_{\mathrm{end}}$ are its first and last
update times. If the block's total terminal weight is at least one,
we define $t_{\mathrm{cross}}$ to be the earliest update in $I$ satisfying
\[
 \sum_{t'=t_{\mathrm{start}}}^{t_{\mathrm{cross}}}h_{t'}\ge1,
 \qquad h_{t'}=\eta_{t'}m_{t'},
\]
and select this update as a candidate. The prefix through
$t_{\mathrm{cross}}$ creates the initial motion needed by
Lemma~\ref{lem:flat}. Together with the block's retention bound, this
supplies both hypotheses of the lemma. Because we chose the earliest
crossing, the updates strictly before $t_{\mathrm{cross}}$ have total
terminal weight less than one.

If the block's total terminal weight is less than one, no crossing
exists. In this case, we select only $t_{\mathrm{end}}$ as a candidate.
The preceding noncandidate updates already have total terminal weight
less than one, so no application of Lemma~\ref{lem:flat} is needed
inside this block.

For a block with crossing $t_{\mathrm{cross}}$, Lemma~\ref{lem:flat} and
$\underline R_T\le1$ give
\[
 \prod_{i=1}^k\psi(\tau_i,t_i)
 \le\Lambda:=16(H+2)\sqrt{1+H}
\]
for every nonempty chronological family
$t_{\mathrm{cross}}\le\tau_1<t_1\le\cdots\le\tau_k<t_k\le t_{\mathrm{end}}$.
Here $H=\sum_{t=0}^{T-1}h_t$ is the total terminal gradient weight.
This is exactly the $\Lambda$-admissibility condition for the suffix
$(t_{\mathrm{cross}},t_{\mathrm{end}}]$, so the algebraic results of
Section~\ref{sec:local} now apply.
In addition to $t_{\mathrm{cross}}$, mark the dyadic crossings
from Lemma~\ref{lem:bands} and $t_{\mathrm{end}}$ if it is not already a candidate.
Each block then has at most $K$ candidates and at most $K$ 2-admissible
runs, where
\begin{equation}\label{eq:band-count}
K:=\frac{\log\Lambda}{\log2}+2\le51\log T.
\end{equation}
The bound $H\le144T^6$ makes the candidate count logarithmic;
Appendix~\ref{app:accounting-constants} verifies the explicit factor $51$.

By Lemma~\ref{lem:local}, each nonempty run $J\subseteq I$ of length $|J|$
has terminal-weight mass
\begin{equation}\label{eq:runmass}
\sum_{t\in J}h_t\le d_I\sum_{t\in J}\eta_t
\le C_0d_I|J|^{\alpha-1}.
\end{equation}
Every noncandidate update belongs either to a prefix of $h$-mass less
than one or to a controlled 2-admissible run, as illustrated in
Figure~\ref{fig:candidate-block}. The runs are disjoint, so their total
length is at most $T$.

\begin{figure}[htbp]
\centering
\begin{tikzpicture}[x=1cm,y=1cm,font=\small]
\node[anchor=west] at (0,2.05) {(a) Consecutive retention blocks (Lemma~\ref{lem:blocks})};
\draw[fill=black!7] (0,0.9) rectangle (5.8,1.55);
\draw[fill=linkblue!9] (6.05,0.9) rectangle (12.4,1.55);
\node at (2.9,1.225) {Block $I$: retention $\ge1/2$};
\node at (9.225,1.225) {Next block: retention $\ge1/2$};
\fill (5.8,0.9) circle (2.5pt);
\fill (12.4,0.9) circle (2.5pt);
\node[anchor=north] at (5.8,0.82) {$t_{\mathrm{end}}$};
\node[anchor=north] at (12.4,0.82) {$t_{\mathrm{next}}$};
\draw (5.8,0.32) -- (5.8,0.15) -- (12.4,0.15) -- (12.4,0.32);
\node[anchor=north,align=center] at (9.1,0.05)
 {$\rho_{t_{\mathrm{end}},t_{\mathrm{next}}}<1/2$\\
  $m_{t_{\mathrm{end}}}/2\le t_{\mathrm{next}}-t_{\mathrm{end}}$};
\end{tikzpicture}

\medskip
\begin{tikzpicture}[x=1cm,y=1cm,font=\small]
\node[anchor=west] at (0,2.2) {(b) Candidates inside one block (Section~\ref{sec:block-candidates})};
\draw (0,1.05) -- (0,1.2) -- (12.4,1.2) -- (12.4,1.05);
\node at (6.2,1.58) {Block $I=[t_{\mathrm{start}},t_{\mathrm{end}}]$: retention at least $1/2$, memory $m_t\le d_I$};
\draw[fill=black!7] (0,0) rectangle (2.55,0.8);
\node[align=center] at (1.275,0.4) {Prefix\\$h$-mass $<1$};
\draw[fill=linkblue!9] (3.25,0) rectangle (6.95,0.8);
\node[align=center] at (5.1,0.4) {2-admissible run $J$\\$h$-mass $\le C_0d_I|J|^{\alpha-1}$};
\draw[fill=linkblue!9] (7.65,0) rectangle (11.35,0.8);
\node[align=center] at (9.5,0.4) {2-admissible run $J$\\$h$-mass $\le C_0d_I|J|^{\alpha-1}$};
\draw (2.55,0.4) -- (3.25,0.4);
\draw (6.95,0.4) -- (7.65,0.4);
\draw (11.35,0.4) -- (12.05,0.4);
\fill (2.9,0.4) circle (2.5pt);
\fill (7.3,0.4) circle (2.5pt);
\fill (12.05,0.4) circle (2.5pt);
\node[anchor=north,align=center] at (2.9,-0.12) {First unit-mass\\candidate $t_{\mathrm{cross}}$};
\node[anchor=north,align=center] at (7.3,-0.12) {Dyadic crossing\\candidate};
\node[anchor=north,align=center] at (12.05,-0.12) {Endpoint\\candidate $t_{\mathrm{end}}$};
\end{tikzpicture}
\caption{From retention blocks to candidates.
Figure~\ref{fig:candidate-block}.a shows why the memory budget is bounded:
each block retains at least half the remaining displacement, but the
interval between consecutive right endpoints retains less than half.
Its length therefore bounds half the earlier endpoint's memory.
Figure~\ref{fig:candidate-block}.b shows how one block's guarantees
are used. The prefix through the first candidate supplies unit terminal
weight; retention makes the suffix $\Lambda$-admissible. Marking dyadic
crossings leaves the blue 2-admissible runs, whose step-size bounds
are multiplied by $d_I$ to give the displayed terminal-weight bounds.
Dots mark candidates, and shaded intervals contain noncandidate
updates. Runs may be empty, and candidates may coincide. A block with
no unit-mass crossing has only its endpoint as a candidate.}
\label{fig:candidate-block}
\end{figure}
\FloatBarrier

\Needspace{13\baselineskip}
\section{Global gap cost and the main theorem}\label{sec:accounting}
The candidate construction exhausts the noncandidate mass: every
noncandidate update lies in a prefix of mass less than one or in a
controlled run.
Consequently, the gap charges have exactly four sources: fixed constants,
candidate memory penalties, pre-crossing terminal-weight mass, and
admissible-run mass. The first three use the candidate count, memory
budget, and unit-mass threshold. The fourth combines the local exponent
with the global memory and length budgets. The following proposition
sums these contributions into the cost needed for
Proposition~\ref{prop:embed}.
\begin{proposition}[Global gap cost]\label{prop:global-cost}
If $\underline R_T\le1$, the constructed candidate list satisfies
\begin{equation}\label{eq:gap-target}
\mathcal D(t_1,\ldots,t_L)\le C_{\mathrm{gap}} T^\alpha\log T
\end{equation}
for an absolute constant $C_{\mathrm{gap}}$.
\end{proposition}

\begin{proof}
Write $\mathcal D=\mathcal D(t_1,\ldots,t_L)$ and form the charges
\eqref{eq:embed}. Because every block endpoint is a candidate,
noncandidate mass in each gap splits into the prefixes and runs
identified above. For each block $I$, let $P_I$ denote its noncandidate
prefix, and let $J$ range over its 2-admissible runs. Empty prefixes
and runs contribute zero.

\Needspace{13\baselineskip}
Subadditivity of the $2/3$ power gives the following four-term bound:
\[
\begin{aligned}
\mathcal D^{2/3}=\sum_{j=0}^L g_j^{2/3}
\le{}&\underbrace{2^{2/3}+L}_{\text{fixed constants}}
 +\underbrace{2^{2/3}\sum_{\text{candidates }t}m_t^{2/3}}
             _{\text{candidate memory}}\\
&+\underbrace{\sum_I\left(\sum_{t\in P_I}h_t\right)^{2/3}}
             _{\text{pre-crossing terminal weight}}
 +\underbrace{\sum_I\sum_{\text{runs }J\subseteq I}
                   \left(\sum_{t\in J}h_t\right)^{2/3}}
             _{\text{admissible-run terminal weight}}.
\end{aligned}
\]
The first two terms come from the constant and memory allowances in
\eqref{eq:embed}; the last two account for all noncandidate updates.
We now bound these four terms in order.

\proofstep{Constants.}
At most $T$ blocks and $K$ candidates per block give $L\le KT$.
The constant terms contribute $2^{2/3}+L\le2KT$, since $K\ge2$.

\proofstep{Memory terms.}
Each candidate in $I$ has $m_t\le d_I$. Since $d_I\ge1$ and
$\sum_I d_I\le3T$, their contribution is bounded by
\[
2^{2/3}\sum_{\text{candidates }t}m_t^{2/3}
\le2K\sum_I d_I^{2/3}
\le2K\sum_I d_I\le6KT.
\]

\proofstep{Pre-crossing mass.}
Each block has at most one noncandidate prefix before its first
unit-mass crossing, or before its right endpoint if there is no
crossing. Each such prefix has weight less than one, so their
contribution is at most $T\le KT$.

\proofstep{Admissible runs.}
Write $N_I=\sum_{\text{runs }J\subseteq I}|J|$ for the total number
of updates in the at most $K$ runs in block $I$.
Concavity and the block memory bound in \eqref{eq:runmass} give
\[
\sum_{J\subseteq I}\left(\sum_{t\in J}h_t\right)^{2/3}
\le C_0^{2/3}d_I^{2/3}K^{1-2(\alpha-1)/3} N_I^{2(\alpha-1)/3}.
\]
Here the selection exponent $2/3$ from \eqref{eq:gap-cost} meets the
local mass exponent $\alpha-1$ from Lemma~\ref{lem:local}. The resulting
memory and length exponents are $2/3$ and $2(\alpha-1)/3$, whose sum
is $2\alpha/3>1$. The retention blocks supply the global memory
budget, and disjointness supplies the length budget, giving
\begin{align*}
\sum_I d_I^{2/3}N_I^{2(\alpha-1)/3}
&\le\sum_I(d_I+N_I)^{2\alpha/3}\\
&\le\left(\sum_I d_I+\sum_I N_I\right)^{2\alpha/3}\le(4T)^{2\alpha/3}.
\end{align*}
The outer $3/2$ power in $\mathcal D$ turns the resulting
$T^{2\alpha/3}$ into $T^\alpha$. Thus the selection recursion,
local drift bound, and block memory budget jointly determine the final
exponent.

Combining the four contributions gives
\begin{equation}\label{eq:gap-power-global}
\sum_{j=0}^L g_j^{2/3}
\le9KT+C_0^{2/3}K^{1-2(\alpha-1)/3}(4T)^{2\alpha/3}
\le C_{\mathrm{gap}}^{2/3}T^{2\alpha/3}(\log T)^{2/3}
\end{equation}
for an absolute cost constant $C_{\mathrm{gap}}$. The last inequality uses $K\le51\log T$,
$0\le1-2(\alpha-1)/3\le2/3$, and $\log T\le T^{2\alpha-3}/(2\alpha-3)$;
Appendix~\ref{app:accounting-constants} records the constants.
Taking the $3/2$ power proves \eqref{eq:gap-target} with the choice of $C_{\mathrm{gap}}$ given there.
\end{proof}

\Needspace{12\baselineskip}
We can now combine the candidate cost bound with the static-chain
certificate to prove the main result, restated here for convenience.
\begin{mainrestatement}
There is an absolute constant $c>0$ such that, for every integer
$T\ge2$ and every predetermined schedule
$\eta\in[0,\infty)^T$, $\beta\in[0,1)^T$, there exist a dimension
$d\le T+1$, a differentiable convex $1$-smooth Huber-chain objective
$f:\mathbb R^d\to\mathbb R$, a minimizer $x^\star$, and an
initial point $x_0$ with $\|x_0-x^\star\|_2\le1$, for which the
Heavy-Ball iteration \eqref{eq:P.1}, initialized with $x_{-1}=x_0$,
satisfies
\[
 f(x_T)-f(x^\star)\ge\frac{c}{T^\alpha\log T},
 \qquad \alpha=\frac{1+\sqrt5}{2}.
\]
The objective is fixed before the run and may depend on the schedule.
\end{mainrestatement}

\begin{proof}[Proof of Theorem~\ref{thm:main}]
If $\underline R_T\le1$, Propositions~\ref{prop:embed} and
\ref{prop:global-cost}, together with Lemma~\ref{lem:static}, give
\[
R_T\ge\underline R_T\ge\frac{1}{8\mathcal D}
\ge\frac{1}{8C_{\mathrm{gap}}T^\alpha\log T},
\]
realized by a single static Huber chain. If $\underline R_T>1$,
Lemma~\ref{lem:static} supplies the witness directly. Since
$T^\alpha\log T\ge1$ for $T\ge2$, the choice $c=1/[8(C_{\mathrm{gap}}+1)]$ works
in both cases.

Rescaling as $f(x)=Lr^2\widetilde f((x-x^\star)/r)$ transforms the step
sizes to $L\eta_t$ and leaves the momenta unchanged, giving the factor $Lr^2$.
\end{proof}

\Needspace{15\baselineskip}
\section{Conclusion}
We show that arbitrary predetermined, horizon-dependent tuning is
insufficient to give the classical Heavy-Ball method Nesterov's $O(T^{-2})$
last-iterate guarantee on smooth convex objectives. For every such
schedule, the worst-case error is
$\Omega(T^{-\alpha}/\log T)$, where $\alpha=(1+\sqrt5)/2$.
This obstruction is realized by a single static objective fixed before
the run, with initialization distance at most one and zero initial velocity.

At the same time, the Heavy-Ball method contains predetermined gradient descent as
the special case $\beta_t=0$, and the Silver stepsize schedule achieves
$O(T^{-s})$ with $s=\log_2(1+\sqrt2)=1.271553\ldots$.
Hence, with $R_T$ defined in~\eqref{eq:P.2}, the Silver guarantee and
Theorem~\ref{thm:main} yield
\[
\frac{c}{T^{1.618033\ldots}\log T}
\le\inf_{\eta,\beta}R_T(\eta,\beta)
\le\frac{C}{T^{1.271553\ldots}},
\]
where the infimum ranges over the admissible predetermined schedules.

The optimal rate of the Heavy-Ball method with predetermined schedules therefore remains unresolved.
In light of the GD lower bound reported by \citet{yeliu2026silver}, the
central question is whether predetermined momentum can improve on the
Silver polynomial exponent, or whether the Heavy-Ball lower bound can be
strengthened toward it. Within the predetermined-schedule model,
classical momentum alone is insufficient to recover Nesterov acceleration
on general smooth convex objectives.

\section*{AI disclosure}
The proof builds on the checkpoint-based one-sided Huber-chain construction of \citet{jung2026}, with GPT-6 Astra Pro playing a substantial role in its derivation. The human authors verified the mathematical arguments and revised the exposition for clarity and readability. They take full responsibility for the correctness and content of the paper.

\bibliographystyle{plainnat}
\bibliography{references}

@article{polyak1964,
  author = {Polyak, Boris T.},
  title = {Some methods of speeding up the convergence of iteration methods},
  journal = {USSR Computational Mathematics and Mathematical Physics},
  volume = {4}, number = {5}, pages = {1--17}, year = {1964},
  doi = {10.1016/0041-5553(64)90137-5}
}

@article{nesterov1983,
  author = {Nesterov, Yurii E.},
  title = {A method of solving a convex programming problem with convergence rate {$O(1/k^2)$}},
  journal = {Doklady Akademii Nauk SSSR},
  volume = {269}, number = {3}, pages = {543--547}, year = {1983},
  url = {https://www.mathnet.ru/eng/dan46009}
}

@article{bubeck2015,
  author = {Bubeck, S{\'e}bastien},
  title = {Convex Optimization: Algorithms and Complexity},
  journal = {Foundations and Trends in Machine Learning},
  volume = {8}, number = {3--4}, pages = {231--357}, year = {2015},
  doi = {10.1561/2200000050},
  url = {https://arxiv.org/abs/1405.4980}
}

@article{goujaud2025,
  author = {Goujaud, Baptiste and Taylor, Adrien and Dieuleveut, Aymeric},
  title = {Provable non-accelerations of the heavy-ball method},
  journal = {Mathematical Programming}, year = {2025},
  doi = {10.1007/s10107-025-02269-2},
  url = {https://arxiv.org/abs/2307.11291}
}

@inproceedings{ghadimi2015,
  author = {Ghadimi, Euhanna and Feyzmahdavian, Hamid Reza and Johansson, Mikael},
  title = {Global convergence of the {Heavy-ball} method for convex optimization},
  booktitle = {2015 European Control Conference (ECC)},
  pages = {310--315}, year = {2015},
  doi = {10.1109/ECC.2015.7330562},
  url = {https://arxiv.org/abs/1412.7457}
}

@article{sun2019,
  author = {Sun, Tao and Yin, Penghang and Li, Dongsheng and Huang, Chun and Guan, Lei and Jiang, Hao},
  title = {Non-Ergodic Convergence Analysis of {Heavy-Ball} Algorithms},
  journal = {Proceedings of the AAAI Conference on Artificial Intelligence},
  volume = {33}, number = {1}, pages = {5033--5040}, year = {2019},
  doi = {10.1609/aaai.v33i01.33015033},
  url = {https://ojs.aaai.org/index.php/AAAI/article/view/4435}
}

@inproceedings{sebbouh2021,
  author = {Sebbouh, Othmane and Gower, Robert M. and Defazio, Aaron},
  title = {Almost sure convergence rates for {Stochastic Gradient Descent} and {Stochastic Heavy Ball}},
  booktitle = {Proceedings of the Thirty Fourth Conference on Learning Theory},
  series = {Proceedings of Machine Learning Research},
  volume = {134}, pages = {3935--3971}, year = {2021},
  publisher = {PMLR},
  url = {https://proceedings.mlr.press/v134/sebbouh21a.html}
}

@article{lessard2016,
  author = {Lessard, Laurent and Recht, Benjamin and Packard, Andrew},
  title = {Analysis and Design of Optimization Algorithms via Integral Quadratic Constraints},
  journal = {SIAM Journal on Optimization},
  volume = {26}, number = {1}, pages = {57--95}, year = {2016},
  doi = {10.1137/15M1009597},
  url = {https://arxiv.org/abs/1408.3595}
}

@article{silver2025,
  author = {Altschuler, Jason M. and Parrilo, Pablo A.},
  title = {Acceleration by stepsize hedging: {Silver Stepsize Schedule} for smooth convex optimization},
  journal = {Mathematical Programming},
  volume = {213}, pages = {1105--1118}, year = {2025},
  doi = {10.1007/s10107-024-02164-2},
  url = {https://arxiv.org/abs/2309.16530}
}

@misc{jung2026,
  author = {Jung, Minchan and Cho, Hanseul and Yun, Chulhee},
  title = {Stronger Lower Bounds for {(Non-)Anytime} Acceleration of Gradient Descent},
  year = {2026}, howpublished = {arXiv:2609.04032v1},
  url = {https://arxiv.org/abs/2609.04032v1}
}

@article{grimmer2024,
  author = {Grimmer, Benjamin},
  title = {Provably Faster Gradient Descent via Long Steps},
  journal = {SIAM Journal on Optimization},
  volume = {34}, number = {3}, pages = {2588--2608}, year = {2024},
  doi = {10.1137/23M1588408},
  url = {https://arxiv.org/abs/2307.06324}
}

@misc{grimmer2023,
  author = {Grimmer, Benjamin and Shu, Kevin and Wang, Alex L.},
  title = {Accelerated Gradient Descent via Long Steps},
  year = {2023}, howpublished = {arXiv:2309.09961},
  url = {https://arxiv.org/abs/2309.09961}
}

@article{grimmer2025,
  author = {Grimmer, Benjamin and Shu, Kevin and Wang, Alex L.},
  title = {Accelerated Objective Gap and Gradient Norm Convergence for Gradient Descent via Long Steps},
  journal = {INFORMS Journal on Optimization},
  volume = {7}, number = {2}, pages = {156--169}, year = {2025},
  doi = {10.1287/ijoo.2024.0057},
  url = {https://arxiv.org/abs/2403.14045}
}

@misc{ma2026,
  author = {Ma, Jianhao and Chen, Yuxin},
  title = {A lower bound for stepsize-based acceleration of gradient descent},
  year = {2026}, howpublished = {arXiv:2608.10418},
  url = {https://arxiv.org/abs/2608.10418}
}

@misc{tsai2026nonanytime,
  author = {Tsai, Chung-En},
  title = {An Improved Lower Bound for Non-Anytime Gradient Descent},
  year = {2026}, howpublished = {Blog post},
  url = {https://chungentsai.github.io/gd-lower-bounds.html}
}

@misc{ye2026,
  author = {Ye, Yuhan and Liu, Kaizhao},
  title = {Improved Gradient Descent Lower Bounds Beyond {Nesterov}},
  year = {2026}, howpublished = {arXiv:2609.02855v2},
  url = {https://arxiv.org/abs/2609.02855v2}
}

@misc{yeliu2026silver,
  author = {Ye, Yuhan and Liu, Kaizhao},
  title = {The {Silver} Rate Is (Almost) Tight},
  year = {2026},
  howpublished = {Blog post, September 7},
  note = {Accessed September 8, 2026},
  url = {https://yeyuhanyyh.github.io/gd-silver-rate/}
}
\appendix
\clearpage
\section{Verification of the static construction}\label{app:realization}
This appendix verifies static realization by combining the starting
response with warm transfers. It then derives the immediate-path bound
for selection and the retention-normalized delayed bound for local
step-size control.
\subsection{Smoothness of the chain}\label{app:smoothness}
For nonnegative offsets and positive thresholds, \eqref{eq:chain} is
convex, nonnegative, differentiable, and vanishes at zero. Since
$0\le\phi_\delta''\le1$ on each smooth piece, every piecewise Hessian
is positive semidefinite. For a test vector $z\in\mathbb R^d$,
its quadratic form is bounded by
\[
\frac14\sum_{i=1}^{d-1}(z_i-z_{i+1})^2+\frac12z_d^2
\le\|z\|^2.
\]
The inequality follows from $(z_i-z_{i+1})^2\le2z_i^2+2z_{i+1}^2$;
it also holds for the scalar chain with an empty link sum.
Integrating along line segments gives the same Lipschitz bound across
breakpoints. Thus the chain is $1$-smooth and, with $x_0=e_1$, has a
minimizer at initialization distance one.

\subsection{Starting response}
The first link must create motion from the initial displacement.
To quantify its response per unit gradient input, run the scalar
recurrence from zero position and velocity with constant gradient $-1$.
Each update adds
$\eta_t$ to velocity. Let $F_t$ record remaining displacement after
update $t$, and $\Delta_t$ accumulated displacement before query $t$:
\begin{equation}\label{eq:response}
F_t=\sum_{t'=0}^t h_{t'}\rho_{t',t},\qquad
\Delta_t=\sum_{t'<t}\frac{F_{t'}}{m_{t'}}.
\end{equation}
The actual velocity after update $t$ is $F_t/m_t$.
Unrolling the recurrence using \eqref{eq:memory-def} gives
\[
F_t\ge h_t,\qquad 0\le\Delta_t\le\sum_{t'<t}h_{t'}.
\]
The accumulated displacement limits the threshold that keeps the
starting link saturated. With that threshold, the link produces
remaining displacement
\begin{equation}\label{eq:start-gain}
G_{\mathrm{start}}(t)=\frac{F_t}{2(2+\Delta_t)}.
\end{equation}

\subsection{One warm transfer}
A warm link must stay saturated while generating outgoing motion.
We therefore track incoming motion, margin correction, and outgoing
response separately. For a warm transfer $(\tau,t]$, $S_{\tau,t}$ measures frontier motion from
query $\tau$ to query $t$ due to unit post-update velocity at $\tau$.
The response $A_{\tau,t}$ measures interference from intervening unit
gradient inputs; it determines how much of the link margin is lost:

\begin{equation}\label{eq:interval-response}
S_{\tau,t}=m_{\tau}(1-\rho_{\tau,t}),\qquad
A_{\tau,t}=\sum_{\tau<t'<t}h_{t'}(1-\rho_{t',t}).
\end{equation}
The outgoing response $W_{\tau,t}$ is the next coordinate's remaining
displacement after update $t$ under unit gradient inputs on $(\tau,t]$:
\[
W_{\tau,t}=\sum_{\tau<t'\le t}h_{t'}\rho_{t',t}.
\]
For a saturated link, the input magnitude is $\delta/4$.
We have $S_{\tau,t}\ge1$ and $W_{\tau,t}\ge h_t$. The warm-transfer factor
$G(\tau,t)$ converts incoming remaining displacement into outgoing remaining
displacement:
\begin{equation}\label{eq:gains}
G(\tau,t)=\frac{W_{\tau,t}}{4m_{\tau}}
\min\left\{1,\frac{2S_{\tau,t}}{2+A_{\tau,t}}\right\}.
\end{equation}
The minimum enforces saturation at queries $\tau+1,\ldots,t$.
For incoming velocity $q$, the outgoing remaining displacement is at least
$G(\tau,t)m_{\tau}q$.

\subsection{Static realization of a path}
In this appendix, a path is indexed from its starting checkpoint $t_0$:
$k$ counts the subsequent warm transfers, so there are $k+1$ checkpoints
and the final coordinate is $k+2$. Section~\ref{sec:motion-loss} instead counts all
checkpoints from $t_1$; its checkpoint count is one larger than this
appendix's warm-transfer count.
For transfer $i$ in a path, use installation time $\tau_i$ and
checkpoint time $t_i$; its previous checkpoint is $t_{i-1}$.
Thus the local notation $t_{\mathrm{prev}}\le\tau<t$ becomes
$t_{i-1}\le\tau_i<t_i$. Each added link must preserve the earlier trajectory. The
following lemma realizes all the path factors in one fixed objective.

\begin{lemma}[Static realization]\label{lem:static-exact}
For every chronological path
\[
 0\le t_0\le \tau_1<t_1\le\cdots\le \tau_k<t_k<T,
\]
there is a single static Huber chain, depending on the schedule and path
and fixed before the run, with terminal error at least
\begin{equation}\label{eq:cert}
 \frac{G_{\mathrm{start}}(t_0)^2}{4(1+\sum_{t'>t_k}h_{t'})}
 \prod_{i=1}^k\bigl[\rho_{t_{i-1},\tau_i}G(\tau_i,t_i)\bigr]^2.
\end{equation}
The case $k=0$ uses just the starting checkpoint $0\le t_0<T$ and an
empty product. The empty-path value $1/[4(1+H)]$ is also realizable.
The certified lower bound $\underline R_T$ introduced in Section~\ref{sec:static}
is the maximum of these finitely many values, including the empty path. Then
$R_T\ge\underline R_T$, with every witness of dimension at most $T+1$.
\end{lemma}

\begin{proof}[Proof of Lemma~\ref{lem:static-exact}]
Use the objective \eqref{eq:chain}, whose smoothness was verified in
Appendix~\ref{app:smoothness}. We first consider paths with positive certificate
value; the scalar construction at the end also covers every zero-valued
path. Positivity ensures that the velocities used to choose thresholds
below are positive. We choose offsets and thresholds in the following order.

\proofstep{1. Create velocity from the initial displacement.}
Set $c_1=0$ and $\delta_1=2/(2+\Delta_{t_0})$.
With later links absent, the first margin at query $t\le t_0$ is at least
$1-\delta_1\Delta_t/2\ge\delta_1$, because its own gradient contribution in the margin
recurrence has magnitude at most $\delta_1/2$. The link therefore saturates at
every such query. Coordinate $2$ has post-update velocity
$q=\delta_1F_{t_0}/(4m_{t_0})$, so
$m_{t_0}q=G_{\mathrm{start}}(t_0)$.

\proofstep{2. Place an inactive link at its installation query.}
Fix the incoming frontier coordinate $i$ and preceding checkpoint $t_{\mathrm{prev}}$.
Its saved velocity and its velocity at installation query $\tau\ge t_{\mathrm{prev}}$ are
\[
q:=x_i(t_{\mathrm{prev}}+1)-x_i(t_{\mathrm{prev}}),\qquad
q':=x_i(\tau+1)-x_i(\tau).
\]
Once $i$, $t_{\mathrm{prev}}$, and $\tau$ are fixed, $q$ and $q'$ denote these two saved
values throughout this transfer; neither denotes a time-independent
velocity of the trajectory. They are components of the full update
vectors. Positive certificates give $q>0$. Before adding
its outgoing link, coordinate $i$ receives only nonnegative velocity
increments from its left link. Its position is nondecreasing, and all
coordinates to its right are zero and at rest. Set $c_i=x_i(\tau)$.
For any threshold, the new link has zero gradient at every query
through $\tau$, preserving both the earlier trajectory and update $\tau$.
Thus the saved velocity $q$ and installation velocity $q'$ remain
well defined when the new link is added, and they satisfy
\[
m_{\tau}q'\ge\rho_{t_{\mathrm{prev}},\tau}m_{t_{\mathrm{prev}}}q.
\]
Repeating this argument shows that the fixed sum \eqref{eq:chain}
reproduces the recursively constructed prefix.

\proofstep{3. Keep the new link saturated during the warm transfer.} For $t>\tau$, choose
\[
 \delta=q'\min\left\{1,\frac{2S_{\tau,t}}{2+A_{\tau,t}}\right\}.
\]
While the following link is absent, bound the link margin below using the scalar comparison sequence
$z_{\tau}=0$, $z_{\tau+1}=q'$ and
\[
z_{t'+1}-z_{t'}=\beta_{t'}(z_{t'}-z_{t'-1})-\eta_{t'}\delta/2\quad(t'>\tau).
\]
Unrolling gives $z_{t'}=q'S_{\tau,t'}-\delta A_{\tau,t'}/2$ for $t'>\tau$.
The actual margin is at least $z_{t'}$: the magnitude of its own link's gradient contribution is at most
$\delta/2$, and the left link contributes a nonnegative increment. Once an increment
of $z$ is nonpositive, all later increments are nonpositive. Hence
$z_{t'}$ first increases and then decreases, and its minimum on
$\tau+1,\ldots,t$ is at an endpoint. The two endpoint inequalities are
\[
z_{\tau+1}=q'\ge\delta,\qquad
z_t=q'S_{\tau,t}-\delta A_{\tau,t}/2\ge\delta,
\]
where the second follows from
$\delta(1+A_{\tau,t}/2)\le q'S_{\tau,t}$. All required queries therefore saturate. The next coordinate's remaining displacement is exactly
\[
 \delta W_{\tau,t}/4=(m_{\tau}q')G(\tau,t)\ge(m_{t_{\mathrm{prev}}}q)\rho_{t_{\mathrm{prev}},\tau}G(\tau,t).
\]
At this checkpoint, coordinate $i+1$ becomes the next frontier, with
saved velocity $x_{i+1}(t+1)-x_{i+1}(t)$. Earlier coordinates can keep
moving; the symbol $q$ is reused only for the next specified transfer.

\proofstep{4. Convert the final remaining displacement into error.}
At the final checkpoint $t_k$, set the terminal offset to $x_d(t_k)$.
Until this penalty is added, the last coordinate receives only
nonnegative force from its left link and is nondecreasing. The penalty
therefore has zero gradient at every query through $t_k$, preserving
the prefix and its saved final velocity. Let
$q=x_d(t_k+1)-x_d(t_k)$ be the last coordinate's velocity.
Its post-update margin also equals $q$. Without later forces, this
velocity would give terminal margin $m_{t_k}q$; we use this remaining
displacement directly rather than introduce a new symbol.
The terminal penalty can oppose it with gradient at most $\delta/2$,
while input from the left can only help. Its final margin is therefore
at least $m_{t_k}q-\delta\sum_{t'>t_k}h_{t'}/2$.
The global supporting line $\phi_\delta(y)\ge\delta y-\delta^2/2$ gives
\[
f(x_T)\ge\frac{\delta m_{t_k}q}{2}
         -\frac{\delta^2(1+\sum_{t'>t_k}h_{t'})}4.
\]
A larger threshold strengthens both the penalty and its braking force.
Balancing them with $\delta=m_{t_k}q/(1+\sum_{t'>t_k}h_{t'})$ yields
\[
f(x_T)\ge\frac{(m_{t_k}q)^2}{4(1+\sum_{t'>t_k}h_{t'})}.
\]
The squared-motion bound comes from this constructed terminal penalty;
large updates do not imply large error on arbitrary objectives.
Multiplying the starting displacement and the transfer factors proves \eqref{eq:cert}. A path with $k$ warm-transfer intervals needs $k+2\le T+1$ coordinates. Finally, the scalar choice $f=\phi_\delta/2$, $x_0=1$, $\delta=(1+H)^{-1}$ gives the empty value by the same supporting-line calculation and \eqref{eq:identities}.
\end{proof}

\subsection{Immediate-path consequence}\label{app:immediate}
For checkpoint selection, install each link at the preceding checkpoint
and bound the path factors directly in terminal weights and memory.
Multiplicativity gives
\[
A_{\tau,t}\le\frac{S_{\tau,t}}{m_{\tau}}\sum_{\tau<t'<t}h_{t'}.
\]
Using $S_{\tau,t}\ge1$ and $W_{\tau,t}\ge h_t$ in \eqref{eq:gains}, and
the starting-response bounds in \eqref{eq:start-gain}, yields
\begin{equation}\label{eq:immediate}
 G(\tau,t)\ge\frac{h_t}{2(2m_{\tau}+\sum_{\tau<t'<t}h_{t'})},\qquad
 G_{\mathrm{start}}(t)\ge\frac{h_t}{2(2+\sum_{t'<t}h_{t'})}.
\end{equation}
\begin{proof}[Proof of Lemma~\ref{lem:static}]
For a nonempty list $t_1<\cdots<t_k$, start at its first checkpoint
and install each subsequent link at the previous checkpoint. Apply \eqref{eq:immediate} to each factor in
\eqref{eq:cert}. This gives \eqref{eq:immediate-path}; the empty path
and attainment of $\underline R_T$ follow from Lemma~\ref{lem:static-exact}.
\end{proof}

\subsection{Retention normalization and delayed paths}\label{app:flat}
Local schedule control requires a test in the original step sizes.
The following lemma bounds the retention-normalized gain below by
\eqref{eq:flat}, comparing both outgoing response and the displacement
constraint. We then combine these bounds along a delayed path.

\begin{lemma}[Momentum removal with original step sizes]\label{lem:flat-exact}
On every positive-retention interval $(\tau,t]$,
\begin{equation}\label{eq:flat-exact}
 \frac{G(\tau,t)}{\rho_{\tau,t}}\ge\psi(\tau,t).
\end{equation}
The score uses the same original step sizes on every subinterval.
\end{lemma}

\begin{proof}[Proof of Lemma~\ref{lem:flat-exact}]
Fix $\ell=t-\tau$ and write the retained fraction of actual velocity as
\[
r_i=\frac{m_{\tau}}{m_{\tau+i}}\rho_{\tau,\tau+i}\quad(0\le i\le\ell),
\qquad 1=r_0\ge r_1\ge\cdots\ge r_\ell>0.
\]
By \eqref{eq:survival-def}, $r_i/r_{i-1}=\beta_{\tau+i}$; positive retention
makes these ratios positive, and $\beta_{\tau+i}<1$ gives monotonicity.
Write $Y$ for the \emph{normalized outgoing response} and $E$ for
the \emph{original step-size mass} on this interval:
\[
Y=\sum_{i=1}^{\ell}\frac{\eta_{\tau+i}}{r_i},\qquad
E=\sum_{i=1}^{\ell}\eta_{\tau+i}.
\]
Then $W_{\tau,t}/(m_{\tau}\rho_{\tau,t})=Y$, so \eqref{eq:gains} becomes
\[
\frac{G(\tau,t)}{\rho_{\tau,t}}
=\min\left\{\frac Y4,\frac{S_{\tau,t}Y}{2(2+A_{\tau,t})}\right\}.
\]
\proofstep{1. Compare total outgoing response.}
Since $r_i\le1$, we have $Y\ge E$. This bounds the first branch below
by $E/4$.

\proofstep{2. Compare the displacement constraint.}
The saturation branch requires both an upper bound on the margin
correction and a lower bound on its complement in $S_{\tau,t}Y$.
Using $S_{\tau,t}=\sum_{j=0}^{\ell-1}r_j$, the ratios $r_j/r_i$ are at
most one for $j\ge i$ and at least one for $j<i$, giving
\begin{align*}
A_{\tau,t}
&=\sum_{i=1}^{\ell}\frac{\eta_{\tau+i}}{r_i}
     \sum_{j=i}^{\ell-1}r_j
 \le\sum_{i=1}^{\ell}(\ell-i)\eta_{\tau+i},\\
S_{\tau,t}Y-A_{\tau,t}
&=\sum_{i=1}^{\ell}\frac{\eta_{\tau+i}}{r_i}
     \sum_{j=0}^{i-1}r_j
 \ge\sum_{i=1}^{\ell}i\eta_{\tau+i}.
\end{align*}
If $E>2$, then $\sum_{i=1}^{\ell}i\eta_{\tau+i}-2\ge E-2>0$.
We may therefore lower-bound the numerator and upper-bound the
denominator in the positive fraction below:
\begin{align*}
\frac{S_{\tau,t}Y}{2+A_{\tau,t}}
&=1+\frac{S_{\tau,t}Y-A_{\tau,t}-2}{2+A_{\tau,t}}\\
&\ge1+\frac{\sum_{i=1}^{\ell}i\eta_{\tau+i}-2}
                 {2+\sum_{i=1}^{\ell}(\ell-i)\eta_{\tau+i}}
 =\frac{\ell E}{2+\sum_{i=1}^{\ell}(\ell-i)\eta_{\tau+i}}.
\end{align*}
If $E\le2$, that fraction need not be positive. Instead, compare the
saturation branch directly with $E/4$; the second sum gives
\[
2S_{\tau,t}Y-E(2+A_{\tau,t})
=(2-E)A_{\tau,t}+2(S_{\tau,t}Y-A_{\tau,t}-E)\ge0.
\]
In this case both branches of the transfer factor are at least $E/4$. In either case,
the minimum is at least $\psi(\tau,t)$, proving \eqref{eq:flat-exact}.
\end{proof}

\begin{proof}[Proof of Lemma~\ref{lem:flat}]
Use the first installation query $\tau_1$ as the starting checkpoint of a path.
Within $[t_{\mathrm{start}},t_{\mathrm{end}}]$, the retention factor between any two updates is at least $1/2$.
Consequently,
\[
F_{\tau_1}=\sum_{t'\le \tau_1}h_{t'}\rho_{t',\tau_1}
\ge\frac12\sum_{t'=t_{\mathrm{start}}}^{t_{\mathrm{cross}}} h_{t'}\ge\frac12,\qquad
G_{\mathrm{start}}(\tau_1)\ge\frac{1}{4(H+2)}.
\]
Lemma~\ref{lem:flat-exact} gives
$G(\tau_i,t_i)\ge\rho_{\tau_i,t_i}\psi(\tau_i,t_i)$ on each warm-transfer interval.
Multiplicativity of retention combines these factors and all waiting
intervals into $\rho_{\tau_1,t_k}\ge1/2$.
Equation~\eqref{eq:cert}, with $\sum_{t'>t_k}h_{t'}\le H$, now yields
\eqref{eq:delayed-path} with denominator $256(H+2)^2(1+H)$.
\end{proof}

\section{Weighted selection and its schedule-specific consequences}\label{app:capacity}
To prove Proposition~\ref{prop:embed}, we first bound the best squared
payoff for arbitrary candidate weights and positive primitive gaps.
Substituting terminal weights and checking the merged denominators then
gives the schedule-specific proposition. Taking every update as a
candidate also yields the polynomial mass bound.

\paragraph{Relation to the extremal checkpoint argument.}
From \citet[Lemma~3.1 and Appendix~C]{jung2026}, we adapt factorization
at a selected checkpoint, log-submodularity, and the perturbation argument
for a common tight singleton. Here arbitrary candidate weights and
positive primitive gaps merge additively. If $x,y$ are the empty costs
of the two children and $a$ is the selected weight, our empty cost is
$x+y+a$, whereas theirs is $x+y+a-2$ because the left child includes a
fixed offset $2$. This changes the recursion and requires the homogeneous
$2/3$ gap-cost estimate proved below.

\subsection{An abstract selection inequality}
The abstract \emph{candidate weight} $a_i$ becomes the terminal gradient
weight $h_{t_i}$ when applied to a schedule. The \emph{gap cost weight}
$g_i$ has the same role as in Section~\ref{sec:checkpoints}. Consider
their alternating list
\[
g_0,\ a_1,\ g_1,\ a_2,\ldots,a_L,\ g_L,
\qquad g_i>0,\quad a_i\ge0.
\]
For a selected subset $S=\{i_1<\cdots<i_k\}$, merge every omitted weight
and its adjacent gaps. Denote the resulting gaps by $s_1,\ldots,s_k$
before the selected weights and $s_{\rm final}$ after the last one.
To mirror the path certificate, define a multiplicative payoff and
pair it with the \emph{aggregate gap cost} $\mathcal D(g)$, the same
cost as in the main text now applied to the gap list:
\begin{equation}\label{eq:selection}
P(S)=\prod_{j=1}^k\frac{a_{i_j}}{2s_j},\qquad
\mathcal D(g)=\left(\sum_{i=0}^L g_i^{2/3}\right)^{3/2}.
\end{equation}
The empty selection has payoff one and a single final gap containing
the whole list. Although the merged gaps depend on $S$, $\mathcal D(g)$ depends
only on the original gaps.

\begin{lemma}[Weighted capacity]\label{lem:capacity}
For every such list,
\[
\max_S\frac{P(S)^2}{4s_{\rm final}}\ge\frac1{8\mathcal D(g)}.
\]
\end{lemma}
\begin{proof}[Proof of Lemma~\ref{lem:capacity}]
We first bound a cost with one power of the payoff, which factors at
a selected checkpoint. Within this proof, define
\[
V(a;g)=\min_{P(S)>0}\frac{s_{\rm final}}{P(S)}.
\]
To prove $V(a;g)\le \mathcal D(g)$, factor at a selected checkpoint and maximize
over the weights. The nonempty tight selections then share a singleton; splitting
there gives two children, and the $2/3$ inequality closes induction.
Finally, increasing the last primitive gap recovers the squared payoff.
Subscripts on $g$ specify sublists.

\proofstep{1. Split at a selected checkpoint.}
Conditioning on selecting checkpoint $j$ leaves independent selections
on its left and right, with a factor $2/a_j$. Including the empty
selection gives the exact recursion
\begin{equation}\label{eq:split}
 V(a;g)=\min\left\{\sum_i g_i+\sum_i a_i,
       \min_{j:a_j>0}\frac2{a_j}V(a_{<j};g_{<j})V(a_{>j};g_{\ge j})\right\}.
\end{equation}
\proofstep{2. Reduce to a positive maximizing list.}
Induct on $L$, starting with $L=0$, where $V(a;g)=\mathcal D(g)$.
Fix the gaps and maximize $V(a;g)$ over the weights. A maximizer exists:
on the superlevel set $V(a;g)\ge\sum_i g_i$, induction and
\eqref{eq:split} give
\[
a_j\le\frac{2\mathcal D(g_{<j})\mathcal D(g_{\ge j})}{\sum_i g_i}.
\]
The superlevel set is nonempty because all-zero weights give
$V=\sum_i g_i$. The displayed bounds confine it to a bounded box.
The finite minimum extends continuously to zero coordinates: positive
primitive gaps bound every numerator away from zero, so costs selecting
a vanishing coordinate diverge on this box. The superlevel set is
therefore compact, and the maximum is attained.
If a maximizing coordinate is zero, merge its adjacent gaps and use
induction with $(x+y)^{2/3}\le x^{2/3}+y^{2/3}$.
Thus suppose all weights are positive, and abbreviate their maximizing
cost as $V=V(a;g)$.

\proofstep{3. Find a checkpoint selected by every tight nonempty set.}
For a fixed list, write $V_S$ for the \emph{one-power cost of selection}
$S$ and keep $V$ for its minimum:
\[
V_S=\frac{s_{\rm final}}{P(S)}
=\frac{2^{|S|}s_{\rm final}\prod_{j=1}^{|S|}s_j}
       {\prod_{i\in S}a_i}.
\]
Call $S$ \emph{tight} when its cost attains the minimum: $V_S=V$. Inserting a candidate weight $a$ into
a gap with left and right parts $x,y>0$ multiplies the cost by
\[
 \frac{2xy}{a(x+a+y)},
\]
which increases strictly with $x$ and with $y$. Adding other checkpoints
can only shorten these two parts. Insert the elements of
$S_1\setminus S_2$ first into $S_1\cap S_2$ and then, in the same order,
into $S_2$. Multiplying the insertion comparisons gives
\[
V_{S_1}V_{S_2}\ge V_{S_1\cap S_2}V_{S_1\cup S_2}.
\]
If $S_1,S_2$ are disjoint and nonempty, the first insertion into the
empty intersection sees the whole list, whereas insertion into $S_2$
sees a strictly shorter gap. The inequality is therefore strict.

For tight $S_1,S_2$, both terms on the right are at least $V$, while
the left is $V^2$. Their union and intersection are therefore tight.
Two nonempty tight sets cannot be disjoint, by strictness. Some nonempty
set is tight: otherwise a small increase in a weight would improve
the unique tight empty cost without closing the slack of the other
finitely many costs. Repeated intersection now shows that
\[
S_\star=\bigcap_{\substack{S\ne\varnothing\\V_S=V}}S
\]
is nonempty and tight. The empty set is tight as well. If it were not,
decreasing a weight selected by every nonempty tight set would
increase all tight costs, again contradicting maximality.

\proofstep{4. Show that the common tight set is a singleton.}
If $S_\star$ contains two weights with values $x\ge y>0$, replace
them by $x+\epsilon$ and $y-\epsilon+\epsilon^2/(4x)$.
For small $\epsilon>0$, their sum increases and their product decreases:
\[
(y-x)\epsilon-3\epsilon^2/4+\epsilon^3/(4x)<0.
\]
The increased sum improves the empty cost. Every nonempty tight set
selects both coordinates, so its merged gaps are unchanged while the
product of its selected weights decreases. Its cost therefore increases;
the finitely many other costs retain their strict slack for sufficiently
small $\epsilon$.
Thus $S_\star=\{j\}$.

Now let $x,y$ be the empty costs of the two children split at $j$.
Tightness of the empty set and $\{j\}$ gives
\[
 V=x+y+a_j=\frac{2xy}{a_j}
 =\frac{x+y+\sqrt{(x+y)^2+8xy}}2.
\]
Each child minimum is at most its empty cost. Equation~\eqref{eq:split}
and $V=2xy/a_j$ force equality in both children, so induction bounds
$x\le \mathcal D(g_{<j})$ and $y\le \mathcal D(g_{\ge j})$.

\proofstep{5. Close induction with exponent $2/3$.}
It remains to prove the scalar inequality
\begin{equation}\label{eq:power}
 \left(\frac{x+y+\sqrt{(x+y)^2+8xy}}2\right)^{2/3}
 \le x^{2/3}+y^{2/3}.
\end{equation}
By homogeneity, normalize $x=u^{3/2}$, $y=(1-u)^{3/2}$, where
$0\le u\le1$. The inequality $\sqrt u\le(1+u)/2$ and its symmetric
counterpart give
\[
x+y\le1-u(1-u),\qquad
2xy=2[u(1-u)]^{3/2}\le u(1-u).
\]
Thus $x+y+2xy\le1$. The expression inside the parentheses in
\eqref{eq:power} is the nonnegative root of
$z^2-(x+y)z-2xy=0$, so it is at most one.
Applying \eqref{eq:power} to the child bounds proves
$V^{2/3}\le \mathcal D(g_{<j})^{2/3}+\mathcal D(g_{\ge j})^{2/3}=\mathcal D(g)^{2/3}$,
closing the induction.

\proofstep{6. Convert the one-power cost to the squared payoff.}
Add $\lambda>0$ to the last primitive gap charge, and write $V_\lambda$ for
the modified one-power minimum. Each final gap becomes
$s_{\rm final}+\lambda$ while $P(S)$ is unchanged.
For a minimizing selection, $(s_{\rm final}+\lambda)^2
\ge4s_{\rm final}\lambda$ gives
\[
 \max_S\frac{P(S)^2}{4s_{\rm final}}
 \ge\frac{\lambda}{V_\lambda^2}.
\]
Choose $\lambda=\mathcal D(g)$. The one-power bound and concavity give
$V_\lambda\le(\mathcal D(g)^{2/3}+\lambda^{2/3})^{3/2}=2^{3/2}\mathcal D(g)$.
This yields the claimed lower bound $1/(8\mathcal D(g))$.
\end{proof}

\subsection{Embedding the schedule}
To apply the abstract inequality to Heavy Ball, check that merged
primitive charges cover the denominators in the static path bound.

\begin{proof}[Proof of Proposition~\ref{prop:embed}]
Substitute the candidates and gaps from the proposition:
\[
a_j=h_{t_j},\qquad \mathcal D(g)=\mathcal D(t_1,\ldots,t_L),
\]
with gaps given by \eqref{eq:embed}.
For a nonempty selection $S=\{i_1<\cdots<i_k\}$, each merged gap
covers the corresponding dynamic denominator:
\begin{align*}
s_1&\ge2+\sum_{t<t_{i_1}}h_t,\\
s_j&\ge2m_{t_{i_{j-1}}}+\sum_{t_{i_{j-1}}<t<t_{i_j}}h_t
\quad(2\le j\le k),\\
s_{\rm final}&\ge1+\sum_{t>t_{i_k}}h_t.
\end{align*}
Equation~\eqref{eq:immediate} bounds each transfer factor below by
$a_{i_j}/(2s_j)$, so Lemma~\ref{lem:static-exact} realizes at least
$P(S)^2/(4s_{\rm final})$. For the empty selection, the single gap
is at least $1+H$, and the scalar certificate suffices.
Taking the best selection and applying Lemma~\ref{lem:capacity}
proves the claim.
\end{proof}

\subsection{Polynomial bound on total terminal gradient weight}\label{app:mass}
With every update a candidate, the selection proof bounds individual
terminal weights and hence $H$, providing the polynomial product budget
used in Section~\ref{sec:local}.

\begin{proof}[Proof of Corollary~\ref{cor:mass}]
Make every update a candidate in \eqref{eq:embed}; thus
$a_j=h_{j-1}$, $g_0=2$, and $g_j=1+2m_{j-1}$.
Recall the proof-local cost $V=\min_{P(S)>0}s_{\rm final}/P(S)$
from the selection proof. Since
$P(S)^2/(4s_{\rm final})\le\underline R_T\le1$, adding one to the final
gap gives, for every positive payoff,
\[
\frac{s_{\rm final}+1}{P(S)}
\ge\frac{s_{\rm final}+1}{2\sqrt{s_{\rm final}}}\ge1.
\]
Thus the modified minimum $V_1$ is at least one.
Splitting at $a_j>0$ and using the one-power bound $V\le \mathcal D(g)$
proved in Lemma~\ref{lem:capacity} yields
\[
 a_j\le2\mathcal D(g_{\rm left})\mathcal D(g_{{\rm right},1})
 \le2\bigl[3T(T+1)^{3/2}\bigr]^2
 =18T^2(T+1)^3\le144T^5.
\]
Here the right gap list includes the added unit. Every primitive gap charge in
either child is at most $3T$, and each child contains at most $T+1$
gaps, so each child cost is at most $3T(T+1)^{3/2}$. The last inequality
uses $T+1\le2T$. Summing the $T$ weights proves the claim. Zero
weights require no separate argument.
\end{proof}

\section{The local sequence estimates}\label{app:drift}
Retention normalization has reduced the local Heavy-Ball problem to
a one-dimensional sequence extremal problem. We bound step-size mass
under a chronological product budget by proving the recurrence, its
integer induction, and the dyadic decomposition.

The proofs in this appendix require only nonnegative sequence entries,
the formula for $\psi$, and the stated admissibility budget. All
potential and window identities are sequence identities. The connection
to the Huber-chain construction is supplied by the momentum-removal
lemma before these estimates are applied.

\subsection{A prefix product potential}
A prefix product potential will pay for upward drift of window mass
and identify the dyadic crossings. For a $\Lambda$-admissible
sequence, let $\Pi_t$ be the largest product
of $\psi$-scores achievable using chronological
disjoint intervals contained in the first $t$ positions. Its dynamic program is
\begin{equation}\label{eq:potential}
 \Pi_0=1,\qquad \Pi_t=\max\{\Pi_{t-1},\max_{0\le \tau<t}\Pi_{\tau}\psi(\tau,t)\}.
\end{equation}
The first term skips entry $t$ and the second appends a final interval
$(\tau,t]$. Thus, under budget $\Lambda$,
\begin{equation}\label{eq:potentialfacts}
 1\le\Pi_t\le\Lambda,\quad \Pi_t\text{ is nondecreasing},\quad
 \psi(\tau,t)\le\Pi_t/\Pi_{\tau},\quad \eta_t\le4\Pi_t/\Pi_{t-1}.
\end{equation}

\subsection{Proof of the sliding-window recurrence}
Throughout this proof, the sequence is 2-admissible. The prefix potential
$\Pi_t$ has total increase at most one; we use this budget to control
growth of the sliding-window mass.

\begin{proof}[Proof of Lemma~\ref{lem:recurrence}]
\proofstep{1. Measure step-size mass in a sliding window.}
Let $\tau$ be the window's starting time, so its entries have times
$\tau+1,\ldots,\tau+\ell$. Choose parabolic weights so that sliding
the window exposes the two masses appearing in $\psi$:
\[
 Q_{\tau}=\sum_{i=1}^{\ell-1}i(\ell-i)\eta_{\tau+i}
 \qquad(0\le \tau\le N-\ell+1).
\]
Write $S_{\tau}$ for the \emph{ordinary window mass} and $A_{\tau}$ for its
\emph{displacement-weighted mass}:
\[
 S_{\tau}=\sum_{i=1}^{\ell}\eta_{\tau+i},\qquad
 A_{\tau}=\sum_{i=1}^{\ell}(\ell-i)\eta_{\tau+i}
 \qquad(0\le \tau\le N-\ell).
\]
Their exact sliding identity is
\[
 Q_{\tau+1}-Q_{\tau}=\ell S_{\tau}-2A_{\tau}-S_{\tau}.
\]
Every length-$\ell$ window has mass at most $\mathcal S_2(\ell)$.

\proofstep{2. Charge upward drift to the prefix potential.}
If $Q_{\tau}>2\ell^2$, then $S_{\tau}\ge4Q_{\tau}/\ell^2>8$, so the first
branch of $\psi$ exceeds $\Pi_{\tau+\ell}/\Pi_{\tau}\le2$.
Because $\Pi_{\tau}\ge1$,
\[
\frac{\Pi_{\tau+\ell}}{\Pi_{\tau}}
\le1+(\Pi_{\tau+\ell}-\Pi_{\tau}),\qquad
0\le\Pi_{\tau+\ell}-\Pi_{\tau}\le1.
\]
The saturation branch of $\psi$ must therefore control the window, giving
\begin{align*}
\ell S_{\tau}&\le2(2+A_{\tau})\bigl[1+(\Pi_{\tau+\ell}-\Pi_{\tau})\bigr],\\
Q_{\tau+1}-Q_{\tau}&\le4+(4+2A_{\tau})(\Pi_{\tau+\ell}-\Pi_{\tau})-S_{\tau}.
\end{align*}
Using $A_{\tau}\le\ell\mathcal S_2(\ell)$ and
$S_{\tau}\ge4Q_{\tau}/\ell^2$ yields
\[
 Q_{\tau+1}\le(1-4/\ell^2)Q_{\tau}+8
       +2\ell\mathcal S_2(\ell)(\Pi_{\tau+\ell}-\Pi_{\tau}).
\]
If $Q_{\tau}\le2\ell^2$, shared weights grow by at most a factor two and
the entering entry is at most eight, giving
$Q_{\tau+1}\le2Q_{\tau}+8(\ell-1)\le8\ell^2$.
To combine the two regimes, define the \emph{excess window mass} $Z_{\tau}$ above this baseline:
\[
Z_{\tau}=(Q_{\tau}-8\ell^2)_+,\qquad z_+=\max\{z,0\}.
\]
The two estimates imply
\begin{equation}\label{eq:drift}
 Z_{\tau+1}\le(1-2/\ell^2)Z_{\tau}
       +2\ell\mathcal S_2(\ell)(\Pi_{\tau+\ell}-\Pi_{\tau}).
\end{equation}

\proofstep{3. Sum the drift and recover unweighted mass.}
Each increment of $\Pi$ occurs in at most $\ell$ window differences.
Since the total increase of $\Pi$ is at most one,
\[
\sum_{\tau=0}^{N-\ell}(\Pi_{\tau+\ell}-\Pi_{\tau})\le\ell,
\qquad Z_0\le\ell^2\mathcal S_2(\ell)/4.
\]
Summing \eqref{eq:drift} now bounds the total weighted mass:
\[
 \sum_{\tau=0}^{N-\ell}Q_{\tau}
 \le8\ell^2N+\frac98\ell^4\mathcal S_2(\ell).
\]
To recover ordinary mass, sum the weight assigned to each entry across
windows. Outside the first and last $\ell$ positions, this total is
\[
\sum_{i=1}^{\ell-1}i(\ell-i)=\ell(\ell^2-1)/6\ge\ell^3/8.
\]
The two boundary strips have total mass at most $2\mathcal S_2(\ell)$.
Thus
\[
\sum_{i=1}^N\eta_i
\le\frac{64N}{\ell}+(9\ell+2)\mathcal S_2(\ell)
\le\frac{64N}{\ell}+10\ell\mathcal S_2(\ell),
\]
proving \eqref{eq:energyrec}.
\end{proof}

\subsection{The induction with integer window lengths}\label{app:local-induction}
To make the power-law balance uniform, choose an integer window near
the balancing scale and handle short sequences with the singleton bound.

\begin{proof}[Proof of Lemma~\ref{lem:local}]
Apply the recurrence in Lemma~\ref{lem:recurrence}.
Use strong induction and the identities
$2-\alpha=1/(1+\alpha)$ and $\alpha(2-\alpha)=\alpha-1$.
If $N^{2-\alpha}<400$, the singleton bound gives
$\mathcal S_2(N)\le8N<3200N^{\alpha-1}$.
Otherwise take $\ell=\lfloor N^{2-\alpha}/100\rfloor$, so
$4\le\ell\le N/2$ and
$N^{2-\alpha}/200\le\ell\le N^{2-\alpha}/100$.
Equation~\eqref{eq:energyrec} and induction give
\[
 \mathcal S_2(N)\le\left(12800+\frac{C_0}{10}\right)N^{\alpha-1}
 <C_0N^{\alpha-1}.
\]
\end{proof}

\subsection{Dyadic decomposition}\label{app:bands}
For a larger budget, separate entries where the prefix potential's
dyadic level increases. Endpoint ratios will certify that each remaining
run is 2-admissible.

\begin{proof}[Proof of Lemma~\ref{lem:bands}]
Remove an entry $t$ when $\lfloor\log_2\Pi_t\rfloor\ne\lfloor\log_2\Pi_{t-1}\rfloor$. The level only increases, giving the count. Write a remaining run as the entries $\tau+1,\ldots,t$.
The potential has the same dyadic level at $\tau$ and $t$, so
$\Pi_{t}/\Pi_{\tau}<2$. For any chronological interval family within this run,
\eqref{eq:potentialfacts} bounds its product by the product of the
corresponding endpoint ratios. Inserting the ratios across omitted
intervals, which are at least one, bounds that product by $\Pi_{t}/\Pi_{\tau}$.
\end{proof}

\section{Retention blocks and global accounting}
\subsection{The block memory budget}\label{app:blocks}
The backward-greedy construction converts loss of retention between
blocks into a bound on their total memory.

\begin{proof}[Proof of Lemma~\ref{lem:blocks}]
Starting at $t_{\mathrm{end}}=T-1$, choose the smallest $t_{\mathrm{start}}\le t_{\mathrm{end}}$ with
$\rho_{t_{\mathrm{start}},t_{\mathrm{end}}}\ge1/2$, declare $[t_{\mathrm{start}},t_{\mathrm{end}}]$ a block, and repeat with $t_{\mathrm{end}}=t_{\mathrm{start}}-1$.
Multiplicativity gives the within-block retention bound. For consecutive
block endpoints $t_{\mathrm{end}}<t_{\mathrm{next}}$ (the current and next
block endpoints), minimality of the next block's left endpoint gives
$\rho_{t_{\mathrm{end}},t_{\mathrm{next}}}<1/2$. This loss of retention is paid for by elapsed time:
\[
 m_{t_{\mathrm{end}}}/2\le m_{t_{\mathrm{end}}}(1-\rho_{t_{\mathrm{end}},t_{\mathrm{next}}})
=\sum_{t'=t_{\mathrm{end}}}^{t_{\mathrm{next}}-1}\prod_{t''=t_{\mathrm{end}}+1}^{t'}\beta_{t''}\le t_{\mathrm{next}}-t_{\mathrm{end}}.
\]
Write $t_{\mathrm{end},I}$ for the endpoint of block $I$.
Summing over consecutive endpoints and using $m_{T-1}=1$ gives
$\sum_I m_{t_{\mathrm{end},I}}\le2T$. Within a block, $m_t\le1+m_{t+1}$ gives
$d_I\le m_{t_{\mathrm{end},I}}+|I|-1$; summing proves \eqref{eq:memory}.
Zero-retention edges are automatically block boundaries.
\end{proof}

\subsection{Constants in the global accounting}\label{app:accounting-constants}
We verify the candidate count first, then the constants needed to
combine memory and run contributions in the aggregate gap cost.
For the candidate count in \eqref{eq:band-count}, the mass bound gives
\[
H+2\le146T^6,\qquad \sqrt{1+H}\le13T^3,
\qquad \Lambda=16(H+2)\sqrt{1+H}\le32768T^9.
\]
Since $32768=2^{15}$, $\log2\le\log T$, $1/\log2\le2$, and
$2\le3\log T$ for $T\ge2$,
\begin{equation}\label{eq:band-count-log}
K=\frac{\log\Lambda}{\log2}+2\le51\log T.
\end{equation}

For \eqref{eq:gap-power-global}, the same absolute cost constant can
be written without component-constant or exponent aliases:
\[
C_{\mathrm{gap}}=
\left[459(2\alpha-3)^{-1/3}
      +C_0^{2/3}51^{2/3}4^{2\alpha/3}\right]^{3/2}.
\]
Indeed, $\log T\le T^{2\alpha-3}/(2\alpha-3)$ and
\eqref{eq:band-count-log} imply
\[
KT\le51(2\alpha-3)^{-1/3}T^{2\alpha/3}(\log T)^{2/3},
\qquad
K^{1-2(\alpha-1)/3}\le51^{2/3}(\log T)^{2/3}.
\]
The second bound uses $K\ge1$ and $1-2(\alpha-1)/3\le2/3$.
Substitution into the four-contribution bound gives
\[
\sum_{j=0}^L g_j^{2/3}
\le C_{\mathrm{gap}}^{2/3}T^{2\alpha/3}(\log T)^{2/3},
\qquad \mathcal D\le C_{\mathrm{gap}}T^\alpha\log T.
\]

\end{document}